\documentclass{amsart}
\usepackage{amssymb,amscd}
\usepackage{graphicx}

 \newtheorem{theorem}{Theorem}[section]
 \newtheorem{thmA}{Theorem}
 \newtheorem{corollary}[theorem]{Corollary}
 \newtheorem{lemma}[theorem]{Lemma}
 \newtheorem{proposition}[theorem]{Proposition}
 \theoremstyle{definition}
 \newtheorem{definition}[theorem]{Definition}
 
\newtheorem{notation}[theorem]{Notation}
 \theoremstyle{remark}
 \newtheorem{remark}[theorem]{Remark}

 \newtheorem{conjecture}[theorem]{Conjecture}
 
 \numberwithin{equation}{subsection}
 \newtheorem{ack}{Acknowledgment}

\newcommand{\BB}{\text{$\mathcal{B}$}}

\newcommand{\MM}{\text{$\mathcal{M}$}}

\newcommand{\FF}{\text{$\mathcal{F}$}}
\newcommand{\GG}{\text{$\mathcal{G}$}}

\newcommand{\LL}{\text{$\mathcal{L}$}}
\newcommand{\NN}{\text{$\mathcal{N}$}}

\newcommand{\LB}{\text{$\Lambda$}}

\newcommand{\Diff}{\operatorname{Diff}}
\newcommand{\Homeo}{\operatorname{Homeo}}

        \newcommand{\field}[1]{\text{$\mathbb{#1}$}}
        \newcommand{\N}{\field{N}}
        \newcommand{\Z}{\field{Z}}
        
        \newcommand{\R}{\field{R}}
        \newcommand{\C}{\field{C}}

\newdimen\theight
\def\TeXref#1{%
             \leavevmode\vadjust{\setbox0=\hbox{{\tt
                     \quad\quad  {\small \textrm #1}}}%
             \theight=\ht0
             \advance\theight by \lineskip
             \kern -\theight \vbox to
             \theight{\rightline{\rlap{\box0}}%
             \vss}%
             }}%

\begin{document}

\title{Exotic non-leaves with infinitely many ends}

\author{Carlos Meni\~no Cot\'on}
\author{Paul A. Schweitzer, S.J.}

\address{Departamento de Análise, Instituto de Matemática e Estatística\\
Universidade Federal Fluminense\\
Rua Professor Marcos Waldemar de Freitas Reis S/N, Camous Gragoat\'a,
Niteroi, Rio de Janeiro 21941-916, Brazil.}

\email{carlos.meninho@gmail.com}

\address{Departamento de Matematica\\
         Pontificia universidade Cat\'olica do Rio de Janeiro\\
         Marques de S\~ao Vicente 225, Gavea, Rio de Janeiro 22453-900, Brazil. }

\email{paul37sj@gmail.com}

\begin{abstract}
We show that any simply connected topological closed $4$--manifold punctured along any compact, totally disconnected tame subset $\Lambda$ admits a continuum of smoothings which are not diffeomorphic to any leaf of a $C^{1,0}$ codimension one foliation on a compact manifold. This includes the remarkable case of $S^4$ punctured along a tame Cantor set. This is the lowest reasonable regularity for this realization problem. These results come from a new criterion for nonleaves in $C^{1,0}$ regularity. We also include a new criterion for nonleaves in the $C^2$-category.

Some of our smooth nonleaves are ``exotic'', i.e., homeomorphic but not diffeomorphic to leaves of codimension one foliations on a compact manifold, being the first examples in this class.
\end{abstract}

 \maketitle

\section*{Introduction}
An interesting question in foliation theory, which can be traced back to J. Sondow  \cite{Sondow}, is to understand what kind of manifolds can be realized as leaves of foliations on compact manifolds. An open manifold which is realizable as a leaf must satisfy some restrictions. Since the ambient manifold is compact, an open manifold has to accumulate somewhere, and this induces recurrence and ``some periodicity'' on its ends.

In \cite{Menino-Schweitzer} the authors exhibit several families of smooth $4$--manifolds with finitely many ends which are homeomorphic to leaves but not diffeomorphic to any leaf of any $C^2$ codimension $1$ foliation on a compact manifold. These examples include topologically simple manifolds such as some exotic $\R^4$'s and $S^3\times\R$'s. In the present paper we give examples of smooth $4$-manifolds that cannot be diffeomorphic to leaves of $C^{1,0}$ codimension one foliations. Since there is a great difference between codimension one foliations of classes $C^{1,0}$ and $C^2$, this is a significant improvement over the results of \cite{Menino-Schweitzer}.

As in \cite{Menino-Schweitzer}, we use the specific and stunning characteristics of $4$--manifolds, the unique dimension in which there are exotic euclidean spaces. These exotic structures enjoy some sort of non--periodicity (for instance, for some of them it is known that they cannot be covering spaces \cite{Taylor1}); this led to a conjecture in the 1990's, partially resolved in \cite{Menino-Schweitzer}, that they cannot be realized as leaves of foliations on compact manifolds.

\medskip

\noindent {\bf Historical background.} In the 1980's it was shown by J. Cantwell and L. Conlon \cite{Cantwell-Conlon} that every open surface is homeomorphic (in fact, diffeomorphic) to a leaf of a foliation on every closed $3$--manifold, and the first examples of topological nonleaves were due to E. Ghys~\cite{Ghys1} and  T. Inaba, T. Nishimori, M. Takamura, N. Tsuchiya \cite{JAP}; these are open $3$--manifolds with  highly non--periodic fundamental groups which cannot be homeomorphic to leaves in a codimension one foliation in a compact manifold (\cite{Ghys1} for $C^0$ foliations and \cite{JAP} for those of class $C^2$). Other examples of Riemannian nonleaves were found by O. Attie and S. Hurder~\cite{Attie-Hurder} and others  (\cite{Phillips-Sullivan}, \cite{Janusz}, \cite{Zeghib}, \cite{Schweitzer}, \cite{Schweitzer1}).
On the other hand, every Riemannian manifold with bounded geometry is isometric (in particular, diffeomorphic) to a leaf of a compact lamination \cite{Alvarez-Barral} and any smooth manifold admits a Riemannian metric with bounded geometry \cite{Greene}, so the realization problem is only of interest in the category of foliated manifolds (or at least finite dimensional laminations).

The first main result of this paper is the following.

\begin{thmA} Let $M$ be a closed simply connected topological $4$--manifold. Let $\Lambda\subset M$ be any infinite, totally disconnected, compact tame subspace. Then there exists an uncountable family of smoothings of $M\setminus \Lambda$, such that the elements of this family are pairwise non--diffeomorphic and they are not diffeomorphic to any leaf of any $C^{1,0}$ codimension one foliation on a compact manifold.
\label{Thm1}\end{thmA}

Recall that such a set $\Lambda\subset M$ is tame if and only if, for some metric in $M$ and for every $\epsilon>0$, $\Lambda$ is contained in the union of finitely many disjoint open balls of radius less than $\epsilon$ \cite{Bing}. In this case the complement $M\setminus\Lambda$ is simply connected and its topology does not depend on the tame embedding. There are wild embeddings of the Cantor set in the $4$--sphere with a non simply connected complement and, even if the complement is simply connected, the embedding may fail to be tame (see \cite{Skora} for an example in $S^3$ which can easily be generalized to higher dimensions).

We remark that some of the manifolds considered in Theorem~\ref{Thm1} are homeomorphic to leaves of codimension one foliations. We shall discuss the topological realization of these manifolds in the last section. For instance, the specific case where $\Lambda$ is a tame Cantor set and $M=S^4$ can easily be realized by a suspension of two circle diffeomorphisms generating a free group over $M$ with two $1$--handles attached. The same question for non-tame Cantor sets is open. These are, as far as we know, the first examples of manifolds that are nonleaves in the smooth category in codimension one but are leaves in the topological category; they are what we call {\em exotic nonleaves\/}.

We shall obtain new criteria for nonleaves for both $C^2$ and $C^{1,0}$ foliations in Theorems~\ref{t:C2 nonleaf criterion} and \ref{t:nonleaf criterion}. Theorem~\ref{t:C2 nonleaf criterion} is just a slight generalization of the method of proof used in \cite{Menino-Schweitzer} and it seems to be a corollary of the classical foliation theory of $C^2$ codimension one foliations. Theorem~\ref{t:nonleaf criterion} is the main result of this work as it allows us to construct many new families of $C^{1,0}$-nonleaves. Moreover, some of the examples of nonleaves given in \cite{JAP} are covered by this criterion, yielding an alternative proof that they are nonleaves. Its statement is rather technical, but it has the following simpler corollary:

\begin{thmA}\label{t:simple criterion}
Let $W$ be a smooth simply-connected manifold that is also simply connected at its ends (Definition \ref{d:end1connected}. Assume that $W$ has an infinite set of smoothly non-periodic ends that are not diffeomorphic to any other end. Then $W$ is not diffeomorphic to any leaf of any $C^{1,0}$ codimension one foliation on a compact manifold.
\end{thmA}

The techniques used in this paper follow the same approach of proof as Ghys' arguments in \cite{Ghys1}. For $C^2$ regularity the main tool is the generalized Kopell lemma and the level theory for foliations as explained in \cite{Cantwell-Conlon2, Candel-Conlon}.

In this paper we deal mainly with the case where $\Lambda$ is infinite. The case of finitely many ends was already studied in \cite{Menino-Schweitzer} but we want to stress that there we were only able to obtain smoothings that are nonleaves in the $C^2$-category. In this context we shall only observe that there exist Stein nonleaves in $C^2$ regularity and no nonperiodic exotic $\R^4$ can be realized as a proper leaf of a $C^2$ codimension one foliation (this was not explicitly observed in \cite{Menino-Schweitzer}). 

Some other applications of the techniques used in \cite{Ghys1} combined with the Kopell lemma (see \cite[Theorem 2.8]{Cantwell-Conlon2}), Duminy's theorem (see \cite{Cantwell-Conlon3} for its general version for foliations) and Novikov's theorem allow us to obtain the interesting result that the only way to realize a simply connected and simply connected at infinity noncompact manifolds, such as $\R^n$, $n\geq 2$, topologically as a proper leaf of a codimension one $C^2$ foliation is in a (generalized) Reeb component (see Proposition~\ref{p:Novikov-Duminy-Kopell} and Remark~\ref{r:Novikov-Duminy-Kopell}). We are not sure whether this result had already been observed, but we have not found any explicit reference.

The paper is organized as follows: In Sections 1, 2 and 3 we present two criteria for nonleaves in both regularities $C^2$ and $C^{1,0}$; these sections are independent of the other ones and only use foliation theory. In Section 4 we present some well-known properties of $4$--manifolds with the exotic ends that we shall use. In Section 5 we present our families of nonleaves and show that they satisfy the hypotheses of the previously given criteria, so this will prove Theorem~\ref{Thm1}. We briefly discuss the topological realization of our nonleaves in Section 6.

\section{Necessary conditions to be leaves}  \label{s:Yinfty} 
Before constructing our families of nonleaves we shall give two criteria for manifolds to be nonleaves in both the $C^2$ and $C^{1,0}$ categories. Recall that a $C^{1,0}$ foliation is a foliation that has a continuous integrable plane field. 

\begin{definition}
A foliated (i.e., leaf preserving) map $h: (M,\FF)\to(N,\GG)$ between $C^{1,0}$ foliations is said to be of class $C^{1,0}$ if it is continuous, its restriction to each leaf is $C^1$, and there are atlases of $M$ and $N$ in which every first partial derivative in the leaf coordinates is also continuous as a function on $M$.
\end{definition}

In order to study nonleaves, we introduce two relevant properties of open smooth manifolds.

\begin{definition}\label{d:periodic_end} An end of a noncompact smooth manifold $W$ is {\em (smoothly) periodic} if there exists a neighborhood $V\subset M$ of that end and a diffeomorphism  $h: V\to h(V)\subset V$ such that $h^n(V)$ defines the given end (i.e., $\{h^n(V)\}$ is a neighborhood base for the end). Observe that this definition can be also applied to nonisolated ends. An open manifold is said to be {\em (smoothly)  nonperiodic} if at least one of its ends is not periodic in the previous sense.

Two ends of smooth manifolds are {\em diffeomorphic} if they have diffeomorphic neighborhoods. It will always be assumed that the orientation is preserved by the diffeomorphism. Two manifolds each with one end are {\em end--diffeomorphic} if their ends are diffeomorphic.
\label{d:periodic}\end{definition}

A smooth manifold $M$ with just one end and with that end smoothly periodic can be decomposed as $K\cup_\partial N\cup_\partial N\cup_\partial N\cdots$ glued along the boundaries, where $K$ 
is a compact smooth manifold with connected boundary and $N$ is a compact smooth manifold with two boundary components. All these boundary components are diffeomorphic, respecting the orientations. The {\em periodic segment\/} $N$ can be identified with $V\setminus h(\mathring{V})$ for an appropriate choice of $V\subset M$, and then $K=M\setminus \mathring{V}$.

\begin{definition}
A smooth manifold $W$ is said to be  {\em $1$-rigid} if the following condition holds:

Every leaf $L$ of a $C^{1,0}$ transversely oriented codimension one foliation on a compact manifold that is diffeomorphic to $W$ admits an open neighborhood $C^{1,0}$-diffeomorphic to a product foliation $L\times I$ where $I$ is an open interval. 
\end{definition}

\begin{remark}
We can define the concept of $k$-rigidity similarly for leaves of codimension $k$ foliations, but we shall not use this concept.
\end{remark}

The remainder of this section is devoted to proving the next proposition. We follow ideas of Ghys \cite{Ghys1}.

\begin{proposition}\label{p:non-R}
If a leaf of a $C^{1,0}$ transversely oriented codimension one foliation of a compact manifold is diffeomorphic to a $1$-rigid and smoothly non--periodic manifold $W$, then the union of all leaves diffeomorphic to $W$ is an open set each of whose connected components fibers over the circle with the leaves as fibers. The completion of each of these components relative to any Riemannian metric in the ambient manifold is noncompact.\end{proposition}

Observe that there always exists a smooth transverse one-dimensional foliation $\NN$ and a biregular foliated atlas, i.e., one in which each coordinate neighborhood is foliated simultaneously as a product by $\FF$ and $\NN$, obtained by integrating a smooth
transverse vector field (see e.g. \cite{Candel-Conlon}). The transverse coordinate changes are only assumed to be continuous but the leaves of $\FF$ can be taken to be $C^1$ manifolds and the local projection along $\NN$ of one plaque to another plaque in the same chart is a diffeomorphism. 

Our basic tool will be Dippolito's octopus decomposition \cite{Dippolito}. Given a saturated open set $U$ of $(M,\FF)$, let $\hat{U}$ be the completion of $U$ in a Riemannian metric of $M$ restricted to $U$. The inclusion $i: U\to M$ clearly extends to an immersion $i: \hat{U}\to M$, which is at most $2$--to--$1$ on the boundary leaves of $\hat{U}$. We shall use $\partial^\tau$ and $\partial^\pitchfork$ to denote tangential and transverse boundaries, respectively.

\begin{theorem}[Dippolito octopus decomposition \cite{Candel-Conlon,Dippolito}]\label{t:octopus}
Let $U$ be a connected saturated open set of a codimension one transversely orientable foliation $\FF$ on a compact manifold $M$.  Then there exists a compact submanifold $K$ (the nucleus) with boundary and corners such that
\begin{enumerate}
\item $\partial^\tau K \subset\partial^\tau\hat{U}$;

\item $\partial^\pitchfork K$ is saturated for $i^\ast\NN$;

\item the set $\hat{U}\setminus K$ is the union of finitely many non--compact connected components $B_1,\dots,B_m$ (the arms) with boundary, where each $\overline{B}_i$ is diffeomorphic to a product $\overline{S}_i\times[0,1]$ by a diffeomorphism $\phi_i: \overline{S}_i\times[0,1]\to \overline{B}_i$ such that the leaves of $i^\ast\NN$ exactly match the fibers $\phi_i(\{\ast\}\times[0,1])$;

\item the foliation $i^\ast\FF$ in each $B_i$ is defined as the suspension of a homomorphism from $\pi_1(S_i)$ to the group of homeomorphisms of $[0,1]$. Thus the holonomy in each arm of this decomposition is completely described by the action of $\pi_1(S_i)$ on a common complete transversal.
\end{enumerate}  \end{theorem}

\begin{definition}
A noncompact manifold $W$ is said to be {\em simply connected at its ends} if for every compact set $C\subset W$ there exists another compact set $C'\supset C$ such that for every $x\in W\setminus C'$ the natural inclusion $\pi_1(W\setminus C',x)\to\pi_1(W\setminus C,x)$ is the $0$-map. Note that this is slightly different from the notion of being simply connected at infinity, which requires the triviality of the map $\pi_0(W\setminus C')\to\pi_0(W\setminus C)$, so that $W$ can only have one end. \label{d:end1connected}
\end{definition}

The next Proposition is an adaptation of Dippolito's criterion for $1$-rigidity, which he enunciated in \cite[Theorem 4]{Dippolito} and corrected in \cite[Theorem 4']{Dippolito2}. We thank Takashi Inaba for pointing out this correction to us.

\begin{proposition}\label{p:product_neighborhood} 
Let $L$ be a proper leaf with trivial holonomy in a transversely oriented codimension one foliation $\FF$ on a compact manifold $Z$ with a transverse foliation $\NN$. If $L$ is  simply connected at its ends, then there exists an open $\FF$--saturated neighborhood $U$ of $L$ which is diffeomorphic to $L\times(-1,1)$ by a $C^{1,0}$--diffeomorphism which carries the bifoliation $\FF$ and $\NN$ to the product bifoliation. \end{proposition}

\begin{proof}
Let $U$ be a saturated one-sided neighborhood of $L$ in the octopus decomposition of $\hat U$ with nucleus $K$ and let $C= K\cap L$. By hypothesis there exists a compact submanifold $C' \subset L$ containing $C$ such that for every $x\in C'\setminus C$ the homomorphism $\pi_1(C'\setminus C,x)\to \pi_1(L\setminus C,x)$ defined by the inclusion is trivial. Since $L$ has trivial holonomy the compact set $C'$ has a bifoliated one-sided neighborhood $N\subset\hat U$. On each arm $B_i$ this bifoliated neighborhood determines a bifoliated $C^{1,0}$-diffeormophism $B_i \cap N \to (S_i \cap C') \times [0,t_i)$ where $S_i = B_i\cap L$ and $t_i \in (0,1]$. Since every loop in $S_i\setminus C'$ based at a point $x_i\in S_i\setminus C'$ becomes contractible in $S_i =B_i\cap L\setminus C$, the bifoliated product for $B_i\cap N$ extends to a bifoliated $C^{1,0}$-diffeomorphism $N_i\to S_i\times [0,t_i)$ where $N_i$ is a one-sided neighborhood of $S_i$ in $B_i$. Gluing together these homeomorphisms produces a one-sided bifoliated neighborhood $N\cup \bigcup_i N_i$ of $L$. The same argument works on the other side, and the union of the two one-sided neighborhoods is the desired two-sided bifoliated neighborhood of $L$.
\end{proof}


\begin{corollary}\label{c:end1connected}  If $L$ is a simply connected proper leaf of a $C^{1,0}$ codimension one transversely oriented foliation on a compact manifold and $L$ is simply connected at its ends, then it admits an open neighborhood $C^{1,0}$ diffeomorphic to the product foliation $L\times (0,1)$.
\end{corollary}
\begin{proof}
Since $L$ is simply connected, it has trivial holonomy, so Proposition \ref{p:product_neighborhood} applies.
\end{proof}

\begin{remark}\label{r:rigidity reduction}
Let $W$ be an open topological manifold obtained by puncturing a simply connected closed manifold of dimension $n\geq 3$ along a tame, closed and totally disconnected subset. It is an easy exercise to check that $W$ is simply connected and simply connected at its ends.
If, moreover, $W$ does not admit a realization as a nonproper leaf of any codimension one foliation (e.g., see Proposition \ref{p:proper} below), then Corollary \ref{c:end1connected} shows that any smoothing of $W$ is $1$-rigid.
\end{remark}

Let $\FF$ be a transversely oriented $C^{1,0}$ codimension one foliation on a compact manifold and let $W$ be a 
smoothly nonperiodic (in the sense of Definition~\ref{d:periodic_end}) and $1$-rigid
manifold. Let $\Omega_1$  be a connected component of the union of leaves of $\FF$ diffeomorphic to $W$. By the definition of $1$-rigidity, $\Omega_1$ is an open set on which the restriction $\FF_{|{\Omega_1}}$ is defined by a locally trivial fibration, so its leaf space is homeomorphic to a one-dimensional manifold, which must be Hausdorff 
since there is a well-defined local flow on $\Omega_1$ along leaves of $\NN$ that permutes the leaves of $\FF$.

\begin{lemma}\label{l:non-compact}
The completed manifold  $\widehat{\Omega}_1$ is not compact.
\end{lemma}
\begin{proof}From the proof of \cite[Lemma 2.16]{Menino-Schweitzer}, if  $\widehat{\Omega}_1$ is compact then every end of a leaf in the interior of  $\widehat{\Omega}_1$ must be isolated and smoothly periodic, in contradiction to the nonperiodicity of $W$.
\end{proof}

Following the approach of Ghys \cite{Ghys1}, we have a dichotomy: the leaf space of $\FF_{|\Omega_1}$, which is a connected Hausdorff one--dimensional manifold, must be either $\R$ or $S^1$.

\begin{proposition}\label{p:leafspace}
	The leaf space of $\FF_{|\Omega_1}$ cannot be $\R$, so $\Omega_1$ must fiber over the circle with the leaves as fibers.
\end{proposition}
\begin{proof}
The proof is the same as the proof of Proposition 2.17 in \cite{Menino-Schweitzer}, which is also an adaptation of Lemma 4.6 in \cite{Ghys1}.
\end{proof}

\begin{remark}
Since $\Omega_1$ fibers over the circle, it follows that there exists on each leaf $L\subset \Omega_1$ a well-defined non-trivial global diffeomorphism, $m: L\to L$ which is defined by ``following'' the transverse direction, i.e., for any $x\in L$ there exists the transverse flowline $\NN_x$ through $x$; then $L\cap\NN_x$ is a discrete set in $\NN_x$ and $m(x)$ is defined to be the next point of $L\cap \NN_x$ after $x$ in the positive direction. This is just the {\em monodromy\/} of the fibration.
\end{remark}

\begin{definition} \label{d:E^+} Let $E^+\subset \Omega_1$ be the set of points $x$ for which $\{m^n(x)\}_{n\in\N}$ is unbounded and $E^-\subset \Omega_1$ the set of points for which $\{m^{-n}(x)\}_{n\in\N}$ is unbounded.
\end{definition}

\begin{proposition}\label{p:monodromy}

Let $\Omega_1$ be a connected open saturated set of a $C^{1,0}$ codimension one foliation on a compact manifold such that it fibers over the circle with the leaves as fibers and $\widehat{\Omega}_1$ is not compact. Then there exists a compact set $C\subset L$ such that $L\setminus C\subset E^+\cup E^-$.
\end{proposition}

\begin{proof}
	In the octopus decomposition of  $\widehat{\Omega}_1$, any point in the arms is in both $E^+$ and $E^-$ since in the arms the flow lines of $\NN$ are just intervals connecting boundary leaves of $\FF_{\widehat{\Omega}_1}$.
On each tangential boundary component of the nucleus, the monodromy is contracting or expanding, so in sufficiently small neighborhoods of  these boundary components one of the sequences $m^n(x)$ or $m^{-n}(x)$ is necessarily unbounded, so $x\in E^+\cup E^-$. Let $C$ be the set formed by removing the arms and small open neighborhoods of the tangential boundary components of the nucleus from $L$.

It just remains to show that $C$ is compact. If this were not the case, then $L$ would have transverse accumulation points in the interior of $K$ and, therefore, the leaf passing through one of these accumulation points would not admit a product neighborhood, contradicting the fact that $\Omega_1$ is a fibration with the leaves as fibers. 
\end{proof}

This completes the Proof of Proposition \ref{p:non-R}.

\section{A criterion for $C^2$ nonleaves}\label{s:nonleaf}  

This section is devoted to proving the next Theorem, which can be seen as a criterion for identifying $C^2$ nonleaves. 

\begin{theorem}\label{t:C2 nonleaf criterion}
Let $W$ be a $1$-rigid smooth manifold. Assume that there exists a smoothly nonperiodic end $\mathbf{e}$ of $W$ which is diffeomorphic to at most finitely many other ends of $W$. Then $W$ is not diffeomorphic to any leaf of any transversely oriented $C^2$ codimension one foliation on a compact manifold.
\end{theorem}

Now we require the higher smoothness $C^2$, so we should explain what we gain by this assumption.

\begin{definition}\label{d:trivial_inf}
Let $U$ be a saturated open set of a $C^{1,0}$ foliation $\FF$ on a compact manifold and let $L\subset U$ be a leaf. Then $L$ is said to be {\em trivial at infinity} for $U$ if there exists a Dippolito octopus decomposition $\hat{U}=K\cup B_1\cup\dots\cup B_n$ with total transversals $T_i$ for each $B_i$, such that, for every $i$, $L\cap T_i$ consists of fixed points for every element of the total holonomy group of the foliated fiber bundle $(B_i,\FF_{|B_i})$.
\end{definition}

The next theorem is a consequence of the so-called {\em generalized Kopell lemma\/} for foliations which can be found in \cite{Cantwell-Conlon2} and in \cite[Proof of Theorem 8.1.26, p. 182]{Candel-Conlon}.

\begin{theorem}\cite{Cantwell-Conlon2}\label{t:Cantwell-Conlon2}
Let $\FF$ be a transversely oriented codimension one $C^2$ foliation on a compact manifold.  Then every proper leaf $L$ is trivial at infinity for every saturated open set containing $L$.
\end{theorem}

We shall also use the following technical Lemma (see Definition \ref{d:E^+}).

\begin{lemma}\label{l:good block}
Let $\Omega$ be a connected saturated open set of a $C^2$ codimension one foliation on a compact manifold and suppose that $\Omega$ is a fibration over the circle with the leaves as fibers.  Then for every end $\mathbf{e}$ of any leaf $L$ in $\Omega$ there exists a (connected) neighborhood $N_\mathbf{e}$ of $\mathbf{e}$  such that either $N_\mathbf{e}\subset E^-$ or $N_\mathbf{e}\subset E^+$.
\end{lemma}

As a corollary we get the proof of Theorem~\ref{t:C2 nonleaf criterion}.

\begin{proof}[Proof of Theorem~\ref{t:C2 nonleaf criterion}]
Let $\Omega_1$ be a connected component of the open set formed by all leaves diffeomorphic to $W$. By Proposition \ref{p:leafspace}, $\Omega_1$ fibers over the circle with the leaves as fibers. Let $\mathbf{e}$ be the smoothly   nonperiodic end of $L$ which is diffeomorphic to at most finitely many other ends of $L$ and let $m$ be the monodromy map of $\Omega_1$.
By Lemma~\ref{l:good block}, there is 
a closed neighborhood $N_\mathbf{e}$ of $\mathbf{e}$ contained in $E^+$ or $E^-$; assume that it is contained in $E^+$.

Since this end is diffeomorphic to at most finitely many other ends, it follows that some iterate $m^k$, $k>0$, must fix $\mathbf{e}$, thus $m^k(N_\mathbf{e})\cap N_\mathbf{e}\neq\emptyset$.
Assume that $N_\mathbf{e}$ is the closure of a connected component of $L\setminus C$, where $C$ is the compact set constructed in Proposition~\ref{p:monodromy}. Since $m(E^+)=E^+$, $m^k(N_\mathbf{e})$ cannot meet $C$ and therefore $m^k(N_\mathbf{e})\subset N_\mathbf{e}$. Since $N_\mathbf{e}\subset E^+$ it follows that, for some $\ell\in\N$, $m^{k\cdot\ell}(N_\mathbf{e})\subset \mathring{N}_\mathbf{e}$. Since $\{m^r(\partial N_\mathbf{e})\}_{r\in\N}$ is a sequence of compact sets converging to $\mathbf{e}$, the end $\mathbf{e}$ must be smoothly    periodic, which is a contradiction.
\end{proof}

\begin{proof}[Proof of Lemma~\ref{l:good block}]
Let $\Omega$ satisfy the hypotheses of the Lemma and let $L$ be a leaf in $\Omega$. Then  by Proposition~\ref{p:monodromy} there exists a compact set $C\subset L$ such that $L\setminus C\subset E^+\cup E^-$. 
By Theorem \ref{t:Cantwell-Conlon2} we can choose an octopus decomposition for  $\widehat{\Omega}$ with the properties in Definition~\ref{d:trivial_inf}. 
Thus for each $i$ there exists a closed discrete bi--sequence $\{t_{i,j}\}_{i,j\in\Z}\subset (0,1)$ such that $L\cap B_i=\bigsqcup_{j\in\Z}S_i\times\{t_{i,j}\}$, where the monodromy acts by ``translation,'' $m(x,t_{i,j})= (x,t_{i,j+1})$ for $x\in S_i$. 
The boundary of $\bar S_i$ is a smooth compact $(n-1)$--manifold (where $n$ is the dimension of $L$) which we may assume to be connected; if not, just connect the components by tubes in $S_i$ and enlarge the nucleus by the saturation of these tubes under the transverse flow $\NN_{|\hat{\Omega}}$.

Observe that each set $L \cap B_i$ is contained simultaneously in $E^+$ and $E^-$, since the transverse flow in the arms connects one boundary component of $B_i$ with the other one. Each $S_i \times \{t_{i,j}\}$ is a neighborhood of an end of $L$ contained simultaneously in $E^+$ and $E^-$, so every end in an arm satisfies the conclusions of the Lemma.

Let $C_1, C_2, \dots, C_k$ be the connected components of $\partial^\tau K$. Note that the holonomy of each $C_r$, $r=1,\dots, k$, is nontrivial, and it must be infinite cyclic since $\Omega$ fibers over the circle and so any nontrivial holonomy map cannot have fixed points in $\Omega$. 

Hence for a sufficiently small open neighborhood $U_r$ of $C_r$ in $K$, chosen so that $\partial\overline U_r$ is transverse to $L$ and $\NN_{|\overline U_r}$ is an interval fibration over a neighborhood of $C_r$ in the boundary leaf, every connected component of $L\cap \overline U_r$ spirals toward $C_r$ and is contained in $E^+$ if the transverse orientation of $\FF$ points towards $C_r$, or in $E^-$ in the contrary case. Now $K'=K\setminus \bigcup_r U_r$ is a compact set and so is $L\cap K'$; this follows from the fact that $L$ is a proper leaf that cannot accumulate on any other leaf in $\Omega$ (which is a fibration), so $L$ cannot have transverse accumulation points in the interior of $K$. 

Now $\partial (L\cap \overline U_r)$ is composed of countably many compact boundary components, all but finitely many of which agree with boundary components of the $\overline{L\cap B_i}$'s that are diffeomorphic copies of the $\partial\overline S_i$'s, which are compact $(n-1)$-submanifolds. Let $A_r$ be the union of $L\cap\overline U_r$ with the connected components of the $\overline{L\cap B_i}$'s which meet $\overline U_r$. By construction, $\partial A_r$ consists of finitely many compact components.

It is also clear that each $A_r$ is completely contained in $E_+$ or in $E_-$ since the pieces attached to $L\cap \overline U_r$ are simultaneously contained in $E_+$ and $E_-$ and $\overline U_r$ is contained in one of these two sets by construction.

We choose the $U_r$'s to be thin enough so that $A_r\cap A_s=\emptyset$ for $r\neq s$. This is possible since the triviality at infinity of $L$ implies that a connected component in any $L\cap B_i$ cannot be close to both tangential boundary components of the arm at the same time. Note that the union of the $ A_r$'s and the components of the $L\cap B_i$'s cover $L\setminus K'$.
Therefore every end of $L$ that is not an end in an arm $\overline B_i$ is an end of some $A_r$, and then $A_r$ is a neighborhood of that end  contained in $E_+$ or $E_-$.
\end{proof}

When $W$ has finitely many ends we can improve Theorem~\ref{t:C2 nonleaf criterion}. This is in fact a consequence of the theory of levels, see e.g.  \cite[Chapter 8]{Candel-Conlon}.

\begin{theorem}\label{t:C2 nonleaf finite ends}
Let $W$ be a smooth open manifold with finitely many ends. Assume that some end of $W$ is smoothly non-periodic. Then $W$ cannot be diffeomorphic to a proper leaf of a transversely oriented codimension one $C^2$ foliation on a compact manifold.
\end{theorem}
\begin{proof}
Assume that $W$ is diffeomorphic to a proper leaf $L$ of a $C^2$ codimension one foliation. In particular $L$ is a local minimal set of $\FF$ in the sense of \cite[Definition 8.1.17]{Candel-Conlon}, according to \cite[Proposition 8.1.19]{Candel-Conlon}, and belongs to a finite level by \cite[Corollary 8.3.10]{Candel-Conlon}. By Duminy's theorem (see e.g. \cite{Cantwell-Conlon3}) $L$ cannot contain an exceptional minimal set in its limit set (otherwise $L$ would have infinitely many ends). Since $L$ has finitely many ends, it follows from \cite[Corollary 8.4.7]{Candel-Conlon} that $L$ must lie at depth $1$. Thus its limit set consists only of compact leaves. 

Each end must accumulate on only one compact leaf (otherwise, a leaf at intermediate level would appear between $L$ and the compact leaf) and this accumulation must be periodic. It follows that each end of $L$ is periodic in contradiction to the hypothesis.
\end{proof}

\begin{corollary}
If $\mathbf{R}$ is a nonperiodic exotic $\R^4$ then it cannot be diffeomorphic to any proper leaf of any $C^2$ codimension one foliation on a compact manifold.
\end{corollary}

\begin{remark}\label{r:C2 nonleaf finite ends} The advantage of Theorem~\ref{t:C2 nonleaf finite ends} with respect to Theorem~\ref{t:C2 nonleaf criterion} involves  avoiding the condition of $1$-rigidity, which is often difficult to check. According to Theorem~\ref{t:C2 nonleaf finite ends}, an open manifold with finitely many ends, where at least one of them is not periodic, will be a $C^2$ nonleaf if and only if it cannot be realized as a nonproper leaf. Checking this last property is, in general, much simpler than checking $1$-rigidity.
\end{remark}

\begin{remark}\label{r:C2 nonperiodic proper infinite ends}
The theory of levels allows us to improve Theorem~\ref{t:C2 nonleaf finite ends} to manifolds whose set of ends has finite topological type (i.e. where the $k$th derived set of its endset is finite), see \cite[Chapter 8]{Candel-Conlon} for more details. We want to stress the fact that there exist $C^\infty$ codimension one foliations on compact manifolds with proper leaves that have nonperiodic ends. An explicit example can be found in our forthcoming work \cite{Menino-Schweitzer2}.
\end{remark}

\section{A criterion for $C^{1,0}$ nonleaves}\label{s:nonleaf 0}
In this section we present a new criterion for an open manifold $W$ not to be diffeomorphic (or homeomorphic) to any leaf of any $C^{1,0}$ codimension one foliation on a compact manifold. We shall follow the same method of proof as Ghys in \cite{Ghys1}, but applied to a different class of manifolds. Throughout this section $\xi(W)$ will denote the space of ends of $W$.

\begin{definition}\label{d:k-to-one}
Let $W$ be an open smooth manifold and let $S\subset\xi(W)$ be a subset of ends. We say that $W$ is {\em at most $k$--to--one at $S$} if each end in $S$ is diffeomorphic to at most $k-1$ other ends in $\xi(W)$.
\end{definition}

\begin{lemma}\label{l:fixed ends}
Let $W$ be a manifold that is at most $k$--to--one at some $S\subset \xi(W)$ for some $k\in\N$. Then for any diffeomorphism $h: W\to W$, $h^{k!}$ induces the identity map on the set $S$.
\end{lemma}

\begin{proof}
Since $h$ is a diffeomorphism, it maps ends to ends diffeomorphically. Since an end of $S$ is diffeomorphic to at most $k-1$ other ends, it follows that $h^{k!}$ must fix every end in $S$.
\end{proof}

Now let $\Sigma\subset W$ be a connected bicollared codimension one closed (compact and connected) submanifold of $W$ that separates $W$ into two connected components.

\begin{remark} 
Let $h:W\to W$ be a homeomorphism that fixes each end belonging to a set $S\subset \xi(W)$ and suppose that
$\Sigma\cap h(\Sigma)=\emptyset$. Let $[\Sigma,h(\Sigma)]$ be the closure of the connected component of $W\setminus(\Sigma\sqcup h(\Sigma))$ bounded by $\Sigma$ and $h(\Sigma)$. There are two possibilities for $[\Sigma,h(\Sigma)]$:
\begin{itemize}
\item $h([\Sigma,h(\Sigma)])\cap [\Sigma,h(\Sigma)] =h(\Sigma)$ and $[\Sigma,h(\Sigma)]$ does not contain ends in $S$.

\item  $h([\Sigma,h(\Sigma)])\cap [\Sigma,h(\Sigma)]\neq h(\Sigma)$ and the endset of $[\Sigma,h(\Sigma)]$ contains $S$ since this case is only possible when $h$ interchanges the ends on each component of $W\setminus [\Sigma,h(\Sigma)]$.
\end{itemize}\label{r:homology-id}
\end{remark}

\begin{lemma}\label{l:No Monodromy}
Let $W$ be a noncompact oriented $n$-manifold that is at most $k$--to--one at a discrete subset $S$ of $\xi(W)$ with $\# S > 2$. Suppose that $\Sigma\subset W$ (as above) separates ends in $S$, i.e., both components of $W\setminus\Sigma$ contain ends of $S$. Then there does not exist any orientation preserving homeomorphism $h:W\to W$ such that the set $\bigcup_{r\in\Z} h^{r}(\Sigma)$ is properly embedded in $W$ and $h(\Sigma)\cap\Sigma=\emptyset$.
\end{lemma}
\begin{proof}
Assume such a homeomorphism $h$ exists. Then, passing to an iterate of $h$, by Lemma \ref{l:fixed ends} we can assume that $h$ fixes every end of $S$.
Set $C_r=h^r([\Sigma,h(\Sigma)])$ and let $V_r^0$ (resp., $V_r^1$) be the closure of the connected component of $W\setminus C_r$ bounded by $h^{r}(\Sigma)$ (resp., $h^{r+1}(\Sigma)$).

If $h(C_0)\cap C_0\neq h(\Sigma)$ then $S$ must be completely contained in the endset of $C_0$ (Remark \ref{r:homology-id}), but this cannot occur since $\Sigma$ separates ends in $S$. Thus we must have $h(C_0)\cap C_0=h(\Sigma)$  and consequently $W=V_0^0\cup C_0\cup V_0^1$. Applying $h^r$ for every $r\in\Z$ we get that $W=V_r^0\cup C_r\cup V_r^1$, with disjoint interiors. Furthermore the end set of $C_r$ does not contain any end in $S$ (again by Remark~\ref{r:homology-id}). Since $\Sigma$ (and also any $h^r(\Sigma)$ separate ends in $S$, the endsets of both $V^0_r$ and $V^1_r$ contain ends in $S$. It follows that $V_{r+1}^0=V_r^0\cup_\partial C_r$ and $V_r^1=V_{r+1}^1\cup_\partial C_r$ for all $r\in\Z$, 
for otherwise $h$ would permute the ends in $V_r^1$ with those in $V_{r+1}^0$ and that is not possible since ends of $S$ are fixed by $h$. Let $W(\Sigma,h)=\cdots C_{r}\cup_\partial C_{r+1}\cup_\partial\cdots$ be the codimension $0$ submanifold of $W$ obtained by sequentially gluing the components $C_r$. It is clear that $W(\Sigma,h)$ is an open subset, but since $\bigcup_{r\in\Z}h^r(\Sigma)$ is properly embedded it also follows that this submanifold is closed. Therefore $W=W(\Sigma,h)$.

Since no $C_r$ has an end in $S$, it would follow that $W$ can have at most two ends in $S$: the two ends defined by the bisequence $\{h^r(\Sigma)\}_{r\in\Z}$, but this contradicts the hypothesis that $\# S>2$.
\end{proof}

\begin{lemma}\label{l:No Exit Region}
Let $W$ be a smooth noncompact oriented $n$-manifold 
that is at most $k$--to--one at a subset $S$ of smoothly non-periodic ends in $\xi(W)$ with $\# S > 1$. Let $\Sigma\subset W$ be a connected codimension one closed smooth  submanifold such that $\Sigma$ separates ends in $S$. Then there does not exist any orientation preserving diffeomorphism $h:W\to W$ such that the set $\bigcup_{r\in\N} h^{r}(\Sigma)$ is properly embedded in $W$ and $h(\Sigma)\cap\Sigma=\emptyset$.
\end{lemma}
\begin{proof}
Assume such a homeomorphism $h$ exists. As above, we can assume without loss of generality that $h$ fixes the ends in $S$. We shall use the notations $C_r$ and $V_r^i$ used to prove the previous lemma. Observe also that, as above, $h(C_0)\cap C_0=h(\Sigma)$. Therefore, by a similar argument, it follows that $V^1_0=C_1\cup_\partial C_2\cup_\partial\cdots$, where the endset of each $C_r$ cannot contain any ends in $S$. It follows that $V^1_0$ can only contain one end in $S$, precisely the one defined by the sequence $\{h^{r}(\Sigma)\}_{r\in\N}$. But this end is periodic since the $C_r$'s are diffeomorphic via $h$ and this contradicts the fact that ends in $S$ are nonperiodic.
\end{proof}

The aim of this section is to prove the following theorem:

\begin{theorem}\label{t:nonleaf criterion}
Let $W$ be a $1$-rigid smooth $n$-manifold. Assume that $W$ is at most $k$--to--one at an infinite discrete set of smoothly non-periodic ends $S\subset\xi(W)$. Then $W$ is not diffeomorphic to any leaf $L$ of any transversely oriented $C^{1,0}$ codimension one foliation on a compact manifold.
\end{theorem}

\begin{notation}\label{n:good_octopus}
Let  $L$ be a leaf of a $C^{1,0}$ codimension one foliation of a compact manifold that is  diffeomorphic to an $n$-manifold $W$ that satisfies the hypotheses of Theorem~\ref{t:nonleaf criterion}. By Proposition~\ref{p:non-R} it follows that $L$ is contained in a connected open saturated set $\Omega_1$ whose leaves are diffeomorphic to $W$, which fibers over the circle with the leaves as fibers, and such that its completion $\widehat\Omega_1$ is non-compact. Let $K\cup B_1\cup\dots B_{m_0}$ be a Dippolito decomposition of  $\widehat{\Omega}_1$ (with $m_0>0$) and let $S_i$ be the corresponding base manifold of each arm $B_i$ with boundary $\partial\overline S_i =\Sigma_i$. As in the proof of Lemma~\ref{l:good block}, we can assume that each $\Sigma_i$ is connected.

Recall that the transverse boundary of the nucleus is nonempty since $\widehat{\Omega}_1$ is non-compact and it is a union of the transverse boundaries of the closures of the arms. The monodromy preserves the transverse boundary; moreover, the foliation $\FF$ restricted to $\partial^\pitchfork \overline{B}_i$ is also a suspension of a group $G_i$ of orientation-preserving homeomorphisms of the interval over the compact $(n-1)$--manifold $\Sigma_i$, which is the boundary of the base manifold $\overline{S}_i$. We use the notation $\Sigma_i\times_{G_i} I$ to denote the trace of the foliation on $\partial^\pitchfork \overline{B}_i$, where $I=[0,1]$ and $G_i$ is a subgroup of $\Homeo_+(I)$. Since the leaves of $\FF_{|\Omega_1}$ are proper without holonomy, the same is true for the leaves of $\Sigma_i\times_{G_i} I$. It follows that the total holonomy group $G_i$ must be either trivial or isomorphic to $\Z$. 
\end{notation}

\begin{definition}
Let $L$ be a noncompact manifold and let $P\subset L$ be a proper closed subset. An end $\mathbf{e}\in\xi(L)$ is said to be an {\em end of} $L$ {\em relative to} $P$, if every neighborhood of $\mathbf{e}$ in $L$ meets $P$. Let $\xi(P;L)$ be the set of ends of $L$ relative to $P$.
\end{definition}

\begin{lemma}\label{endsfinite}
If the hypothesis of Theorem \ref{t:nonleaf criterion} holds, then
 the set $\xi(\partial^\pitchfork K\cap L;L)$ is finite for all $L\in\FF_\Omega$.
\end{lemma}

\begin{proof}[Proof of Theorem~\ref{t:nonleaf criterion}]
By Lemma \ref{endsfinite} (which will be proven below) the set $\xi(\partial^\pitchfork K\cap L;L)$ is finite. Thus, since $S$ is infinite, there must exist an end in $S$ (in fact infinitely many) and a neighborhood $N$ of that end that does not meet $\partial^\pitchfork K\cap L$. Since $L\setminus (E^+\cup E^-)$ is compact by Proposition~\ref{p:monodromy}, there must exist a codimension one connected closed submanifold $\Sigma\subset N$ that  separates ends in $S$ and does not meet $L\setminus (E^+\cup E^-)$. Then $\Sigma\subset E^+\cup E^-$, so it is completely contained in $E^+$ or in $E^-$ (otherwise it would meet $\partial^\pitchfork K\cap L$), but this contradicts Lemma~\ref{l:No Exit Region}.
\end{proof}

\begin{lemma}\label{l:relativeends}Suppose that $P$ is a noncompact closed subset of a noncompact manifold $L$ such that each connected component of $P$ is compact and open in $P$ (i.e. isolated as a connected component). If there exists a proper path $\gamma:[0,\infty)\to L$  that  meets every component of $P$, then $\xi(P;L)= 1$.
\end{lemma} 
\begin{proof}
Suppose that there exists an end $\mathbf{e}\in \xi(L)$ relative to $P$ that is different from the end $\mathbf{e}_\gamma$ determined by the proper path $\gamma$.
Let $R$ be a compact region in $L$ such that $\mathbf{e}$ and $\mathbf{e}_\gamma$ are ends of distinct components of $L\setminus R$. Then infinitely many distinct components $P_i$, $i\in {\mathbb N}$, of $P$ must meet $R$. If we choose $x_i\in P_i\cap R$ for each $i$
then there will be a subsequence of $\{x_i\}$ that converges to a point of $P\cap R$,
 but that is impossible since the connected components of $P$ are isolated.
\end{proof}

\begin{proof}[Proof of Lemma \ref{endsfinite}]
Note that $\xi(\partial K\cap L;L)=\bigcup_{i}\xi(\partial \overline{B}_i\cap L;L)$. Since this union is finite, it suffices to show that each $\xi(\partial\overline{B}_i\cap L;L)$ is finite. We divide the proof into various cases.

Case 1:  $G_i$ (defined in Notation \ref{n:good_octopus}) is isomorphic to $\Z$. It follows that the interior leaves of $\Sigma_i\times_{G_i} I$ are cyclic covering spaces of $\Sigma_i$, so the leaves in $L\cap \partial\overline{B}_i$ are open $(n-1)$--submanifolds with exactly two ends. They are proper submanifolds of $L$ by construction. Observe also that the generator of each $G_i$ cannot have fixed points in the interior of $I$, in view of the product structure on a neighborhood of each leaf. Thus its fundamental domain is a closed interval and each component of $L\cap \partial \overline{B}_i$ meets that interval in exactly one point (or two for boundary points of the interval). The product structure near each leaf also implies that $L$ meets this interval in at most finitely many points, so $L\cap \partial \overline{B}_i$ consists of finitely many proper submanifolds, each one with two ends. These are the only ends in $\xi(\partial \overline{B}_i\cap L;L)$, so  $\xi(\partial \overline{B}_i\cap L;L)$ is finite in this case.

Case 2: $G_i$ is trivial. In this case the suspension is foliated as a product $\Sigma_i\times I$. Let $L_t$ be the leaf of $\FF_{|\Omega_1}$ which contains $\Sigma_i\times\{t\}$ and for a fixed $t'\in (0,1)$ let $L = L_{t'}$. It follows from the definition of the monodromy that $L\cap\partial\overline{B}_i=\bigsqcup_{k\in\Z}m^k(\Sigma_i\times\{t'\})$.
Let $X$ be a connected component of $L\cap \overline{B}_i$. 

Case 2a: $X$ is bounded by at least two hypersurfaces of the form $m^r(\Sigma_i\times\{t'\})$, $r\in\Z$. Without loss of generality assume that $\Sigma_i\times\{t'\}$ and $m^{k_i}(\Sigma_i\times\{t'\})$ with $k_i>0$ are boundary components of $X$, and that $m^{\ell}(\Sigma_i\times\{t'\})$ does not meet $\partial X$ for $0<\ell<k_i$.
Let $\gamma$ be a path in $X$ joining some $x\in\Sigma_i\times\{t'\}$ to $m^{k_i}(x)\in m^{k_i}(\Sigma_i\times\{t'\})$. Then the union of the paths $m^{k\cdot k_i}\circ \gamma$, $k\in\Z$, defines a proper biinfinite path in $X$ that meets every component of $\bigsqcup_{k\in\Z}m^{k\cdot k_i}(\Sigma_i\times\{t'\})$.  
It follows from Lemma~\ref{l:relativeends} that $\xi\left(\bigsqcup_{k\in\Z}m^{k\cdot k_i}(\Sigma_i\times\{t'\});L_t\right)\leq 2$ so $\xi(\partial \overline{B}_i\cap L;L)=\xi\left(\bigsqcup_{k\in\Z}m^k(\Sigma_i\times\{t'\});L_t\right)\leq 2k_i$ and therefore is finite.

Case 2b: $\overline{X}$ is a connected component of $\overline{B}_i\cap L$ bounded by just one hypersurface $\Sigma_i\times\{t'\}$. Of course, in this case, every other connected component of $\overline{B}_i\cap L$ is also bounded by one such hypersurface. Note that the ends of the connected components of $\overline{B}_i\cap L$ do not belong to $\xi(\partial \overline{B}_i\cap L;L)$.

We shall use notation from the proof of Lemma~\ref{l:good block}  which we recall now. Let $C_1,\dots,C_{n_0}$ be the connected components of $\partial^\tau K$ and for each $1\leq \ell \leq n_0$ let $U_\ell$ be a small tubular neighborhood of $C_\ell$ in $K$ such that $\partial \overline{U}_\ell$ is transverse to $L$, each $\overline{U}_\ell$ is contained in $E^+$ or $E^-$ and $\overline{U}_j\cap \overline{U}_\ell=\emptyset$ for $j\neq l$.

As in Proposition~\ref{p:monodromy}, let us consider $C=L\setminus (\bigcup_j B_j\cup\bigcup_\ell U_\ell)$, which is a codimension zero compact submanifold of $L$ (with boundary and corners). Since $\overline{U}_j\cap \overline{U}_\ell=\emptyset$ for $j\neq \ell$, any path contained in $L\setminus C$ joining a point from $U_j$ to a point in $U_\ell$ needs to pass through an arm.
All but finitely many components of $\partial \overline{B}_i\cap L$ must be contained in $\bigcup_{\ell=1}^{n_0} U_\ell$. Consequently $\xi(\partial \overline{B}_i\cap L;L)=\bigcup_\ell\xi(\overline{U}_\ell\cap\partial \overline{B}_i\cap L ;L)$. (Also note that for a fixed $i$ only two intersections $\overline{U}_\ell\cap\partial \overline{B}_i\cap L$ are nonempty).

Let $\Sigma_i\times\{t'\}$ be a boundary component of $\partial \overline{B}_i\cap L$ contained in $U_{\ell_0}$ for some $1\leq \ell_0\leq n_0$ and assume that $\overline{U}_{\ell_0}\subset E^+$.

\noindent {\bf Claim.} For every $x\in\Sigma\times\{t'\}$ there must exist a path contained in $E^+\cap L$ joining the point $x\in\Sigma_i\times\{t'\}$ to $m^k(x)\in m^k(\Sigma_i\times\{t'\})$ for some $k>0$.

Let $B_\diamond$ be the union of the arms $B_j$ such that $G_j$ is trivial and each connected component of $\overline{B}_j\cap L$ has a connected and compact boundary. Let $V_{\ell_0}$ be the union of $U_{\ell_0}\cap L$ with the connected components of $B_\diamond\cap L$ whose closures meet $\overline{U}_{\ell_0}$, so it is clear that $\overline{V}_{\ell_0}\subset E^+\cap L$. Observe that by construction each connected component of $\partial \overline{V}_{\ell_0}$ meets $\partial C$ or $\partial \overline{B}_j$ for some $j$ such that either $G_j$ is infinite cyclic or else some (and consequently every) connected component of $\overline{B}_j\cap L$ is bounded by at least two compact boundary components of the form $\Sigma_j\times\{t\}$.

Let $\gamma_k:[0,1]\to L$ be a path joining $m^k(x)\in m^k(\Sigma_i\times\{t'\})$ to $x\in \Sigma_i\times\{t'\}$ for each $k>0$. Assume that all of them meet $L\setminus E^+$ and let $t_k\in (0,1)$ be the minimum value such that $\gamma_k(t_k)\notin V_{\ell_0}$; the existence of $t_k$ is clear since $V_{\ell_0}$ is open and $x\in V_{\ell_0}\subset E^+$.
Suppose that for infinitely many $k$'s we have $\gamma_k(t_k)\in\partial C$. Note that $\gamma_k(t_k)\in\partial \overline{V}_{\ell_0}$ and any component of $\partial \overline{V}_{\ell_0}$ is contained in $E^+\cap L$. 
Now $\partial \overline{V}_{\ell_0}\cap\partial C$ has finitely many connected components, since only finitely many of them meet $\partial \overline{U}_{\ell_0}\cap L$ (by transversality of $\partial \overline{U}_{\ell_0}$ to $L$) and the others are 
boundary components of $C$, which is compact.
Thus, there must exist two paths $\gamma_{k_1}$ and $\gamma_{k_2}$ with $k_1<k_2$  such that $\gamma_{k_1}(t_{k_1})$ and $\gamma_{k_2}(t_{k_2})$ belong to the same component of $\partial \overline{V}_{\ell_0}\cap \partial C$. Let $\lambda$ be a path in that boundary component joining $\gamma_{k_1}(t_{k_1})$ and $\gamma_{k_2}(t_{k_2})$. Let $\gamma_{k_1, k_2}$ be the union of the paths $\gamma_{k_1|[0,t_{k_1}]},\lambda,\gamma^{-1}_{k_2|[-t_{k_2},0]}$ (here $\gamma^{-1}(t)=\gamma(-t)$). This path is clearly contained in $E^+\cap L$ and joins $m^{k_1}(\Sigma_i\times\{t\})$ to $m^{k_2}(\Sigma_i\times\{t\})$. By composing with $m^{-k_1}$ we obtain the desired path.

Otherwise, if for all but finitely many $k$'s we have $\gamma_k(t_k)\notin \partial C$, then, by the construction of $V_{\ell_0}$, there must exist an arm $B_j$ such that $\gamma_k(t_k)\in \overline{B}_j$ for infinitely many $k$'s in $\N$. There are three cases to consider:

\begin{itemize}
\item $G_j\simeq \Z$. In this case, by the same reason as in the first case of this proof, $\partial \overline{B}_j\cap L$ consists of finitely many connected two-ended submanifolds, so there must exist two paths meeting the same component of $\partial \overline{B}_j\cap L$, and then we can construct a path $\gamma_{k_1,k_2}$ in $\overline{U}_{\ell_0}\cap L\subset E^+\cap L$ joining $m^{k_1}(x)$ and $m^{k_2}(x)$ as above for suitable $k_1<k_2$. Composing with $m^{-k_1}$ we get the desired path.

\item $G_j$ is trivial but $L\cap \overline{B}_j$ still consists of finitely many components (this is what happens in Case 2a). In this case there must exist $k_1<k_2$ such that $\gamma_{k_1}(t_{k_1})$ and $\gamma_{k_2}(t_{k_2})$ belong to the same connected component of $L\cap \overline{B}_j$ but possibly to different components of $L\cap\partial \overline{B}_j$. In this case consider a path $\lambda:[0,1]\to \overline{B}_j\cap L$ joining $\gamma_{k_1}(t_{k_1})$ to $\gamma_{k_2}(t_{k_2})$. The union of the paths $\gamma_{|[0,t_{k_1}]}$, $\lambda$, and $\gamma_{|[-t_{k_2},0]}^{-1}$ is a path joining $m^{k_1}(x)$ to $m^{k_2}(x)$ and contained in $(\overline{B}_i\cup\overline{U}_{\ell_0}\cup \overline{B}_j)\cap L$  and therefore is contained in $E^+\cap L$.

\item $G_j$ is trivial and every connected component of $\overline{B}_j\cap L$ is bounded by just one component of the form $\Sigma_j\times\{s\}$. 
In this case $B_j\subset B_\diamond$ so $\gamma_k$ must exit $\overline{B}_j$ at $\gamma_k(t_k)$ and therefore $\gamma_k(t_k)\in\partial C$, but this was assumed to happen for finitely many $\gamma_k$'s.
The Claim is proven. 
\end{itemize}

Thus there must exist a  path $\gamma_{k_0}$ which is completely contained in $E^+\cap L$ and joins $x\in\Sigma_i\times\{t'\}$ to $m^{k_0}(x)\in m^{k_0}(\Sigma_i\times\{t'\})$. The union of the paths $m^k\circ \gamma_{k_0}$, for $k\in\N$ defines a semiinfinite proper path meeting the hypersurfaces $m^{k\cdot k_0}(\Sigma_i\times\{t'\})$, $k\in\N$.
Set $P=\bigsqcup_l m^{k\cdot k_0}(\Sigma_i\times\{t'\})$. By Lemma~\ref{l:relativeends} we have that $\xi(P;L)=1$. It readily follows that $\xi\left(U_{\ell_0}\cap\partial \overline{B}_i\cap L;L\right)= \bigcup_{k=0}^{k_0-1}\xi(m^k(P);L)\leq r_0$ and must be finite. A similar argument works if $\overline{U}_{\ell_0}\subset E^-$. 

This finishes the proof of Lemma \ref{endsfinite} and the proof of Theorem \ref{t:nonleaf criterion} is complete.
\end{proof}

\begin{figure}\label{f:transverse boundary nucleus}
\centering
\includegraphics[scale=0.35]{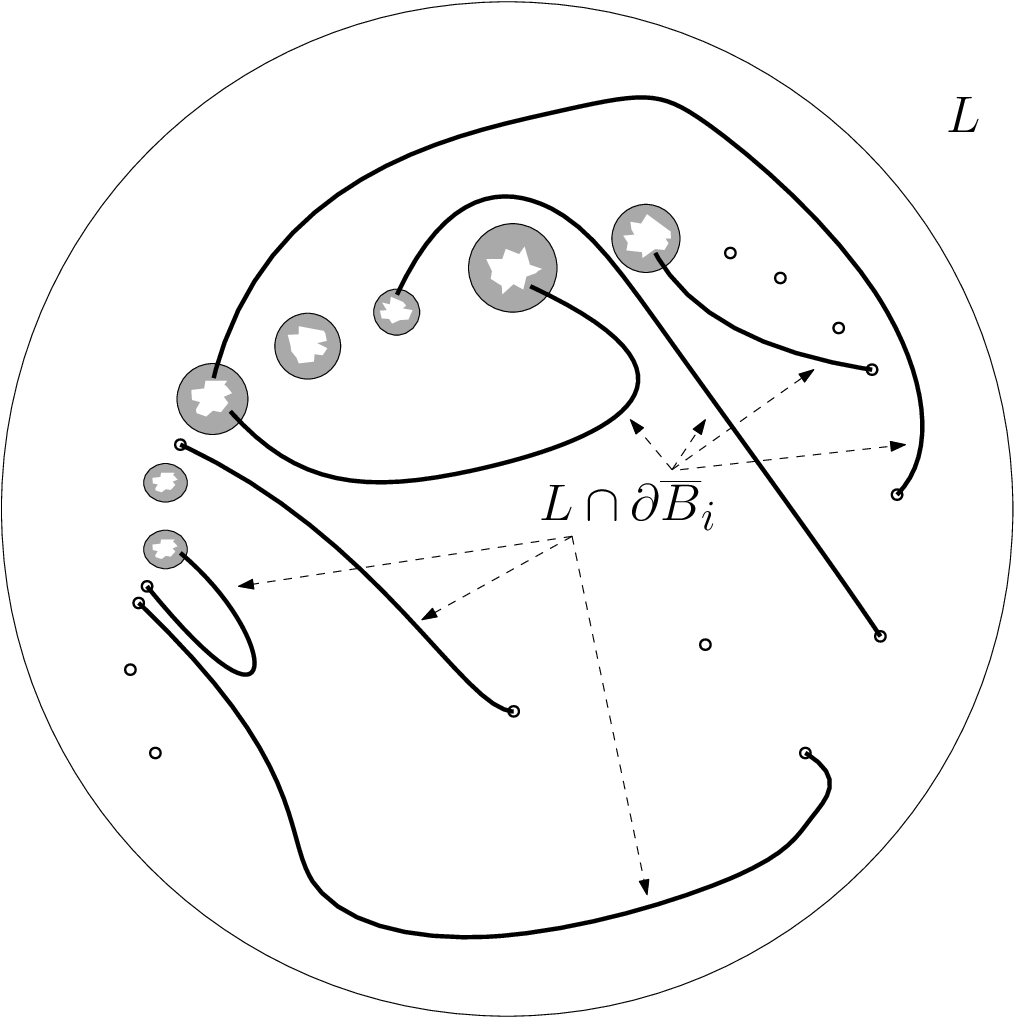}
\includegraphics[scale=0.35]{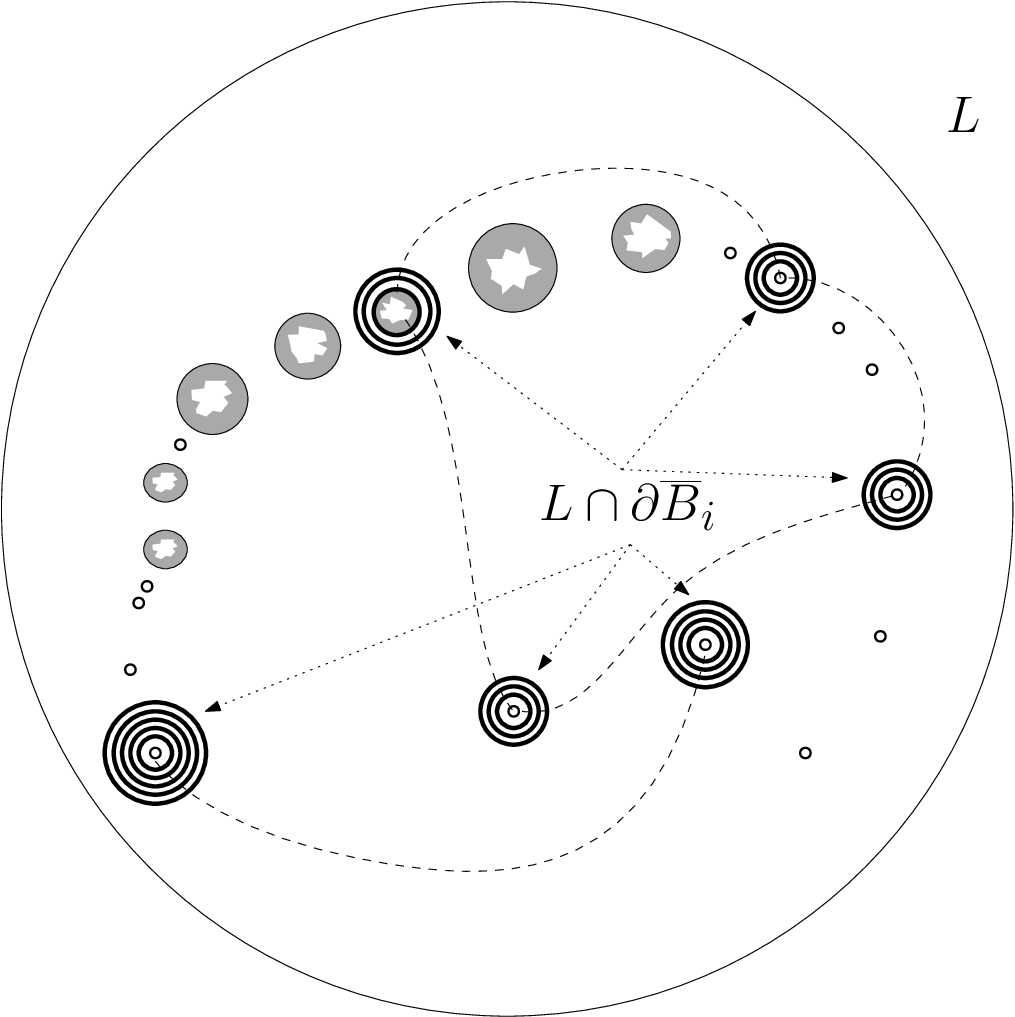}
\caption{A representation of the relative position of $L$ and $\partial \overline{B}_i$, grey circles represent nonperiodic ends, small circles represent other ends. Left: Case $1$ in the proof of Lemma~\ref{endsfinite}, bold lines represent the intersection of $L$ with $\partial \overline{B}_i$, they represent proper submanifolds with two ends. Right: Case $2$ in the same proof, here dashed lines represent proper paths meeting the connected components of $\partial\overline{B}_i\cap L$ which are compact in this case (concentrical bold circles, the centers represent the relative ends). In any case some of the nonperiodic ends, being infinitely many, must avoid $\partial \overline{B}_i\cap L$.}
\end{figure}

\begin{remark}
In \cite{Ghys1} and \cite{JAP}, families of nonleaves obtained as connected sums of lens spaces patterned on an infinite graph are introduced. Each vertex represents a lens space with a fundamental group of prime order and no lens space is repeated in any other vertex. When the graph has infinitely many ends, such a manifold satisfies the conditions of Theorem~\ref{t:nonleaf criterion} ($1$-rigidity is explicitly shown in these papers), thus providing an alternative proof that these are $C^{1,0}$-nonleaves.
\end{remark}

\begin{remark}\label{r:double cover}
Observe that if a leaf of a codimension one foliation is simply connected then there exists a diffeomorphic (homeomorphic) copy of it in the transversely orientable double cover. Thus Theorems~\ref{t:C2 nonleaf criterion} and \ref{t:nonleaf criterion} also apply to simply connected leaves of non transversely oriented foliations.
\end{remark}

\begin{remark}
Topological versions of periodicity and $1$-rigidity can be easily defined in order to study nonleaves of $C^0$ foliations. All the results given in this section (and also in Section \ref{s:Yinfty}) for $C^{1,0}$ codimension one foliations can be directly adapted  to $C^0$ ones, by just changing the words ``smooth'' and ``diffeomorphism'' to ``topological'' and ``homeomorphism''. Given the nature of our next examples, it seems more natural to work only in the smooth category.
\end{remark}

\section{Exotic structures on $\R^4$}   

\subsection{Background} 
We refer to our previous paper \cite{Menino-Schweitzer} for a more detailed background on the necessary tools of $4$--dimensional topology that are needed here. We also refer to the book of R. Gompf and A.I. Stipcsicz \cite{Gompf2} for an extensive guide through the $4$--dimensional world.

According to Freedman's famous theorem \cite{Freedman}, the topology of simply connected closed $4$--manifolds is characterized by their intersection form and their Kirby--Siebemann invariant. For each symmetric bilinear unimodular form over the integers there exists at least one topological simply connected closed $4$--manifold with an isomorphic intersection form \cite{Freedman}, but this is no longer true for the smooth case.

\begin{theorem}[Donaldson \cite{Donaldson}]\label{t:Donaldson}
If $M$ is a closed simply connected smooth $4$--manifold with definite intersection form then its intersection form is isomorphic over the integers to a diagonal form.
\end{theorem}

Since the number of isomorphism classes of definite forms grows at least exponentially with the range, Donaldson's Theorem shows there are many topological simply connected $4$--dimensional manifolds that are not smoothable. On the other hand, noncompact $4$--manifolds are always smoothable \cite{Freedman-Quinn}. There exists a version of Donaldson's Theorem for noncompact smoothly periodic $4$-manifolds due to Taubes \cite{Taubes}. We shall use the following version of Taubes' Theorem which suffices for our purposes.

\begin{theorem}[Taubes \cite{Taubes}] \label{t:Taubes}
Let $M$ be an open smooth simply connected one--ended $4$--manifold with definite intersection form and a smoothly periodic and cylindrical end. Then the intersection pairing on $H_2(M,\Z)$ is isomorphic (over $\Z)$ to a diagonal form.
\end{theorem} 

Here {\em cylindrical} means that the given end admits admits a neighborhood homeomorphic to $S^3\times [0,\infty)$. In the following, if $W$ is a smooth oriented manifold, then $\overline{W}$ will denote the smooth manifold obtained by reversing orientation in $W$.

\begin{definition}[End-sum, see e.g. \cite{Gompf2}]
Let $M$ and $N$ be simply connected oriented open smooth $4$--manifolds with just one end. Choose two smooth properly embedded paths $c_1: [0,\infty)\to M$ and $c_2: [0,\infty)\to N$. Let $V_1$ and $V_2$ be smooth tubular neighborhoods of $c_1([0,\infty))$ and $c_2([0,\infty))$. The boundaries of these neighborhoods are clearly diffeomorphic to $\R^3$ and so we can form a smooth sum by identifying $M\setminus \mathring{V}_1$ and $N\setminus \mathring{V}_2$ along these boundaries so as to produce a manifold with an orientation respecting the orientations of $M$ and $N$. This is called the {\em end--sum} of $M$ and $N$ and it is denoted by $M\natural N = (M\setminus \mathring{V}_1)\bigcup_\partial (N\setminus \mathring{V}_2)$. If $M$ and $N$ are also simply connected at infinity (for example, if the ends are cylindrical), then $c_1$ and $c_2$ are unique up to ambient isotopy and thus the smooth structure of $M\natural N$ does not depend on the choice of the paths.
\end{definition}

\subsection{Large exotica}  

\begin{definition}\cite[Definition 1.11]{Menino-Schweitzer}
Let $\MM_-$ (resp. $\MM_+$) be the the family of smoothings of closed topological $4$--manifolds, $M$, with exactly one puncture so that there exists $s\in\N$ (the positive integers) such that the $s$--fold end--sum $s\natural M\equiv\natural_{i=1}^s M$ is end--diffeomorphic to a smoothing of a once punctured topological simply connected negative (resp. positive) definite but not diagonal $4$--manifold. Set $\MM=\MM_-\cup\MM_+$. These manifolds and their ends are called {\em Taubes-like\/}. Let $\mathcal{S}\subset\MM$ be the subset of Taubes-like manifolds homeomorphic to $\R^4$.	\label{d:Taubes-like}
\end{definition}

The following is an easy but useful condition to show that a manifold is Taubes--like.

\begin{lemma}\cite[Lemma 1.14]{Menino-Schweitzer}\label{l:extension}
Let $Z$ and $Y$ be two smooth $4$--manifolds end--homeomorphic to $\R^4$ so that $Z\subset Y$ and $Y\setminus Z$ is homeomorphic to $S^3\times [0,\infty)$ with topologically bicollared boundary. If $Z\in\MM$ then $Y\in\MM$.
\end{lemma}

\begin{notation}\label{n:good_set}
Let $\mathbf{e}$ be a cylindrical end of a smooth manifold $Z$ and let $X$ be a neighborhood of that end with a homeomorphism $\psi=\psi_\mathbf{e}: X\to S^3\times[0,\infty)$. Let us consider the sets $\mathbf{C}_t=\psi^{-1}([0,t])$ and $\mathbf{K}_t=(M\setminus X)\cup\mathbf{C}_t$ for every $t>0$, and the sets $\mathbf{C}(t;\varepsilon)=\psi^{-1}(S^3\times (t-\varepsilon,t+\varepsilon))$ for $0< \varepsilon<t$, each endowed with the smooth structures induced by the ambient manifold $Z$.
	
Note that the boundary of each $\mathbf{K}_t$ does not need to be smooth, but the homeomorphism $\psi$ can be chosen so that every $\mathbf{K}_t$ is smooth in a neighborhood of a point in its boundary \cite{Quinn}; these will be called {\em admissible} homeomorphisms. In order to avoid confusion, we can use the notations $\mathbf{K}_t^\psi, \mathbf{C}^\psi_t$ whenever the underlying  homeomorphism is not clear from the context.
\end{notation}

The next proposition is an easy corollary of Taubes' Theorem (Theorem~\ref{t:Taubes}).

\begin{proposition}[\cite{Menino-Schweitzer}, Remark 1.16]\label{p:heritage}
Given $Z\in\MM$, let $X$ be a neighborhood of its end with a homeomorphism $\psi: X\to S^3\times [0,\infty)$. Then there exists $r_\psi>0$ such that, for any $t>s>r_\psi$, $\mathring{\mathbf{K}}_t$ is not end--diffeomorphic (preserving the orientation) to $\mathring{\mathbf{K}}_s$.
\end{proposition}

\begin{remark}\label{r:heritage}
Using Lemma \ref{l:extension}, we choose $r_\psi$ so that $\mathring{\mathbf{K}}_t\in\MM$ for all $t>r_\psi$.
Also note that this proposition shows that the underlying topological manifold of any $Z\in \MM$ has a continuum of non-diffeomorphic smooth structures.
\end{remark}

\begin{proposition}
Take $Z\in\MM$ and $\psi$ as in Notation~\ref{n:good_set}. For $t>r_\psi$, $\mathring{\mathbf{C}}_t$ cannot be smoothly embedded in any smooth manifold $S$ homeomorphic to $S^4$.  Moreover, if $Z\in\MM_+$ (resp. $Z\in\MM_-$) then $Z$ cannot be smoothly embedded in any smooth manifold homeomorphic to $k\#\C P^2$ (resp. $k\#\overline{\C P}^2$) for any $k\in\N$.
\label{p:no embedding}
\end{proposition}

\begin{proof} 
Since $Z\in\MM$, there exists $s\in\N$ such that $s\natural Z$ is end--diffeomorphic to a smoothing of a once punctured definite but non--diagonal simply connected manifold $N$. By Remark~\ref{r:heritage}, for all $t>r_\psi$, $N_t:=s\natural\mathring{\mathbf{K}}_t$ is still end--diffeomorphic to another smoothing of a once punctured definite manifold. If a neighborhood of the end of $\mathring{\mathbf{K}}_t$ can be smoothly embedded in $S$ then a neighborhood $X$ of the end of $N_t$ can be smoothly embedded in $s\#S$, which is still homeomorphic to $S^4$. Choose $X$ to be homeomorphic to $S^3\times[0,\infty)$ and let $i: X\to s\# S$ be the above smooth embedding in $s\#S$. Let $A$ be the closed connected component of $(s\#S)\setminus i(X)$ and consider the smooth manifold $W=N_t\cup_i(i(X)\cup A)$ obtained by the identification $i: X\to i(X)$. This is a simply connected smooth closed $4$--manifold with definite non--diagonalizable intersection form, but this contradicts Donaldson's Theorem. A similar argument applies when $Z\in\MM_-$ (resp. $\MM_+$); in this case $s\# S$ will be replaced by the $s$-fold connected sum of smooth manifolds homeomorphic to $k\#\C P^2$ (resp. $k\#\overline{\C P}^2$) and therefore $W$ still has positive (resp. negative) definite and not diagonalizable form.
\end{proof}

\begin{remark}\label{r:no S4 - CP2 maybe}
As a consequence of the above proposition, no Taubes-like exotic $\R^4$ can be embedded in $S^4$. An exotic $\R^4$ with this property is called {\em large\/}. However, there exist Taubes-like exotica in $\MM_-$ that can be smoothly embedded in $\C P^2$ (see \cite[Example 5.10]{Taylor1}).
\end{remark}

\begin{corollary}\label{c:no embedding2}
Take $\mathbf{R}\in\mathcal{S}$ and $\psi$ as in Notation~\ref{n:good_set}. Then for $t>s>r_\psi$, $\mathbf{C}(t;\varepsilon)$ cannot be smoothly embedded (preserving the orientation) in $\mathring{\mathbf{C}}_s$ for any $\varepsilon>0$. \label{Csub}
\end{corollary}
\begin{proof}
Let $j$ be a smooth embedding of $\mathbf{C}(t;\varepsilon)$ into $\mathbf{C}_s$. In any case, $\psi^{-1}(S^3\times\{t+\varepsilon\})$ must be facing toward the unbounded component of $\mathbf{C}_s\setminus j(\mathbf{C}(t;\varepsilon))$, for otherwise $\mathbf{C}(t;\varepsilon)$ could be embedded in (a possibly exotic) $S^4$ since the two boundaries would be bounded by possibly exotic disks, in contradiction to Proposition~\ref{p:no embedding}.  Let $E$ be the closure of the connected component of $\mathbf{C}_{t+\varepsilon}\setminus j(\mathbf{C}(t;\varepsilon))$ located between these two tubes. If $j(\mathbf{C}(t;\varepsilon))$ disconnects the boundary components of $\mathbf{C}_{t+\varepsilon}$ then $E$ has exactly two boundary components and the manifold $$\mathbf{K}_0\cup_\partial j(\mathbf{C}(t;\varepsilon))\cup_\partial E \cup_\partial j(\mathbf{C}(t;\varepsilon))\cup_\partial E \cup_\partial\cdots$$ will be a Taubes-like manifold with a smoothly periodic end, contradicting Theorem~\ref{t:Taubes}.

If $E$ contains $\psi^{-1}(S^3\times\{0\})$ as a boundary component then consider $\mathbf{E}=E\cup_\partial \mathbf{K}_{0}$, which can be used again to produce a periodic segment on a Taubes-like manifold, also giving a contradiction.
\end{proof}

\begin{remark}\label{r:no embedding2}
With the same hypotheses as above, assume that $\Sigma$ is a topologically bicollared sphere in $\mathbf{C}_t\setminus\mathbf{C}_s$ that is homologous to $\psi^{-1}(S^3\times\{0\})$ and let $U$ be a tubular open neighborhood of $\Sigma$ in $\mathbf{C}_t$. Then, by the argument  in Corollary~\ref{c:no embedding2}, $U$ cannot be embedded in $\mathbf{C}_s$. 
\end{remark}

It is said that a closed $4$-manifold has {\em hyperbolic intersection form} if its intersection form is isomorphic to a direct sum of copies of
$$H=\left[\begin{array}{cc}
0 & 1\\
1 & 0
\end{array}\right]\;.$$

\begin{definition}[Taylor's exotica]\label{d:Taylor exotica}
Let $\mathbf{R}$ be a Taubes-like exotic $\R^4$ and let $X$ be a cylindrical neighborhood of its end. Then $\mathbf{R}$ is {\it Taylor-like} if:
\begin{itemize}
\item $\mathbf{R}$ can be smoothly embedded in some closed smooth spin $4$-manifold with hyperbolic intersection form, and

\item there exist an unbounded nondecreasing sequence $\gamma_n\in\N$, $n\in\N$, depending only on $\mathcal{R}$, an admissible homeomorphism $\psi: X\to S^3\times[0,\infty)$, and $r_\psi>0$ such that  no $k\natural \mathring{\mathbf{K}}^\psi_{t}$ with $t>r_\psi$ can be embedded in any smooth closed spin $4$-manifold with hyperbolic intersection form
and second betti number lower than $\gamma_k$.
\end{itemize}
Let $\mathcal{R}\subset\mathcal{S}$ denote the family of Taylor-like exotic $\R^4$'s.
\end{definition}

\begin{remark}
In \cite[Example 5.6]{Taylor1}, Taylor shows that $\mathcal{R}$ is nonempty. Observe that for every $\mathbf{R}\in\mathcal{R}$, $k\natural\mathbf{R}$ is also Taylor-like for any $k\in\N$ but $\infty\natural \mathbf{R}$ cannot be smoothly embeded in any smooth closed spin $4$-manifold with hyperbolic intersection form. In terms of \cite{Taylor1}, the sequence $\gamma_n=2\gamma(n\natural\mathbf{R})-1$ where $\gamma(E)$ is the invariant defined in that work (also called the Taylor-index in \cite{Menino-Schweitzer}), satisfies the condition in the previous definition. 
Note that  $\mathring{\mathbf{K}}_t$ is also Taylor-like for all $t>r_\psi$.

The definition does not depend on the choice of the admissible homeomorphism $\psi$: let $\varphi$ be another admissible homeomorphism from $X$ to $S^3\times[0,\infty)$ and let $\pi:S^3\times[0,\infty)\to [0,\infty)$ be the projection onto the second factor and choose $r_\varphi$ greater than the maximum of $\pi(\varphi\circ\psi^{-1}(S^3\times\{r_\psi\})$; then it follows that for all $t>r_\varphi$ there exists $s>r_\psi$ so that $\mathring{\mathbf{K}}^\psi_s\subset \mathring{\mathbf{K}}^\varphi_t$.\label{r:Taylor-like}
\end{remark}

\begin{lemma}\cite{Taylor1}\label{l:taylor stimate} 
Let $W$ be a smooth open $4$-manifold (possibly with boundary). Take $\mathbf{R}\in\mathcal{R}$ and let $\psi:X\to S^3\times [0,\infty)$ be  an admissible homeomorphism from a neighborhood $X$ of the end of $\mathbf{R}$ to a standard cylinder. 
Let $U\subset X$ be a closed subset homeomorphic to $S^3\times [0,1]$ with bicollared boundary components and such that the bounded component of $X\setminus U$ contains $\mathbf{K}_t$, for some $t>r_\psi$. Suppose that $W$ contains infinitely many pairwise disjoint smooth copies $U_k\ ( k\in\N)$ of $U$ such that $W\setminus U_k$ has a connected component homeomorphic to a ball. Then $W$ cannot be smoothly embedded in any closed smooth spin $4$-manifold with hyperbolic intersection form.
\end{lemma}
\begin{proof}
Assume that $W$ can be smoothly embedded in a closed spin $4$-manifold $M$ with hyperbolic intersection form. For simplicity we shall treat $W$ as a subset of $M$. 

Let $D$ be the union of $U$ with the bounded component of $\mathbf{R}\setminus U$. Let $D_k$ be the union of $U_k$ with the component of $W\setminus U_k$ homeomorphic to a ball (only one of the components can be homeomorphic to a ball for otherwise $U$ would embed in a smooth manifold homeomorphic to $S^4$). 

Moreover, a one-sided tubular neighborhood of $\partial D_k$  in $D_k$ is diffeomorphic to a one-sided tubular neighborhood of $\partial D$ in $D$, for otherwise $U$ would embed in $D_k\cup_U D$ (by a smooth identification preserving the orientation of $U$) and this latter smooth manifold is homeomorphic to $S^4$.

Observe also that $D_i\cap D_j=\emptyset$, for otherwise, if some $D_i\cap D_j\neq\emptyset$ then either $U_i\subset D_j$ and $D_i\subset \mathring{D}_j$ or $U_i\subset W\setminus D_j$ and then $D_j\subset \mathring{D}_i$  (by the disjointness of the $U_k$'s). If it is  the former that holds, then $D_j\setminus \mathring{D}_i$ would be an exotic cylinder with diffeomorphic boundaries that can be used to construct the following periodic smooth manifold (observe that there is just one way to identify boundaries)
$$D\cup_\partial (D_j\setminus \mathring{D}_i)\cup_\partial (D_j\setminus \mathring{D}_i)\cup_\partial\cdots\;,$$ 
but this manifold must be Taubes-like since $D$ contains $\mathbf{K}_t$ for $t>r_\psi$. Thus the $D_k$'s are disjoint.

Let $\gamma_k$ the sequence of integers given in Definition~\ref{d:Taylor exotica} associated to $\mathbf{R}$. Choose $m\in\N$ such that $\gamma_m>\beta_2(M)$. By a smooth surgery, we can replace the disks $D_1,\dots, D_{m}$ by smooth copies of $D$. The effect of this surgery does not change the topology of $M$ but it may change its smooth structure. Let $\hat{M}$ be the smooth manifold obtained by this surgery; then it is still spin with hyperbolic intersection form and same betti numbers since they  depend only on the topology of $M$.

By construction $\mathbf{K}_t\subset D$ and therefore there are $m$ copies of $\mathbf{K}_t$ in $\hat{M}$. But the boundary of $\mathbf{K}_t$ contains a smooth patch and therefore, by \cite[Proposition 2.2]{Taylor1}, $m\natural\mathring{\mathbf{K}}_t$ can be smoothly embedded in $\hat{M}$, but this is a contradiction since $\beta_2(\hat{M})<\gamma_m$.
\end{proof}

\subsection{Small exotica}\label{s:small exotica}$\newline$  

\begin{definition}
An exotic $\R^4$ which can be smoothly embedded in the standard $\R^4$ (or $S^4$) is called {\em small}.
The family of small exotica will be denoted by $\Re$.\end{definition}

\begin{proposition}[\cite{DeMichelis-Freedman}, Corollary 4.1]
There exists an uncountable family of small exotic smooth structures on $\R^4$ which are pairwise non--diffeomorphic. 
\end{proposition}

\begin{corollary}\label{c:good small family}
There exists an uncountable family $\Pi\subset\Re$ of small exotica with the following properties:
\begin{enumerate}
\item no two elements in $\Pi$ are end--diffeomorphic;
and
\item no element in $\Pi$ is smoothly periodic.
\end{enumerate}
\end{corollary}
\begin{proof}
Smoothly periodic manifolds with just one end can be decomposed as $K\cup_\partial N\cup_\partial N\cup_\partial\cdots$ where $K$ is a compact smooth manifold with smooth boundary and $N$ is a compact smooth manifold with two diffeomorphic smooth boundaries (with opposite orientations), one of which coincides with $\partial K$. It is well known that the diffeomorphism classes of compact smooth manifolds with boundary are countable; this is a direct corollary of the finiteness theorems for Riemannian manifolds given by Cheeger in \cite{Cheeger}. Thus, up to diffeomorphism, there are only countably many candidates for $K$ and for $N$. Hence the family of diffeomorphism classes of smoothly periodic exotica is countable.

In a similar way, the set of one--ended smooth manifolds which are end--diffeomorphic to a given one--ended smooth manifold is at most countable. This is explicitly indicated for exotic $\R^4$'s in \cite[Exercise 9.4.13]{Gompf2}, and it follows from the fact that the set of compact codimension zero smooth submanifolds is countable up to isotopy; again, this is a consequence of the finiteness theorems in \cite{Cheeger}.

Since $\Re$ is uncountable it follows that there exists an uncountable subfamily that satisifies both properties.
\end{proof}

\subsection{Nonisolated exotic ends}
Since we are going to deal with manifolds with infinitely many ends we shall need to deal with nonisolated ends. That is the aim of this section.

Let $M$ be a compact manifold and let $\Lambda$ be a closed totally disconnected subset of $M$. Recall that $\Lambda\subset M$ is {\em tame\/} if, for every $\epsilon>0$, $\Lambda$ is contained in the union of finitely many disjoint open balls of radius less than $\epsilon$ \cite[Theorem 1.1]{Bing}.

Another useful condition equivalent to tameness is that there exists an ambient homeomorphism of $M$ which carries $\Lambda$ into a piecewise linear arc in a single chart; it follows that $\Lambda$ can be taken to be a closed subset of the classical ternary Cantor set in the arc, which is identified with the interval $[0,1]$.
If $M$ is noncompact, let $\overline{M}$ be its end compactification. We say that $\Lambda\subset M$ is {\em tame} if $\overline{\Lambda}$ is tame in $\overline{M}$.

\begin{lemma}\label{l:nonisolated smoothings}
Let $\mathbf{R}$ be a Taubes-like exotic $\R^4$'s. Let $\psi$ be an admissible homeomorphism of a neighborhood of its end to $S^3\times [0,\infty)$. Let $t>s>r_\psi$ and let $\Lambda_t\subset\mathring{\mathbf{K}}_t$ and $\Lambda_{s}\subset\mathring{\mathbf{K}}_s$ be closed, totally disconnected and tame subsets accumulating on the ends of $\mathring{\mathbf{K}}_t$ and $\mathring{\mathbf{K}}_s$ respectively. Then the end of $\mathring{\mathbf{K}}_t\setminus\Lambda_t$ induced by $\mathring{\mathbf{K}}_t$ is not end--diffeomorphic to the end of $\mathring{\mathbf{K}}_s\setminus \Lambda_s$ induced by $\mathring{\mathbf{K}}_s$.
\end{lemma}
\begin{proof}
Suppose that $f:N_t\setminus\Lambda_t\to N_s\setminus\Lambda_s$ is a diffeomorphism between neighborhoods of the ends of $\mathring{\mathbf{K}}_t$ and $\mathring{\mathbf{K}}_s$ respectively in the punctured manifolds. Without loss of generality we can assume that $\Lambda_t\subset N_t$. 

By the assumption of tameness, there  exists a smooth properly embedded path $p:[0,\infty)\to \mathring{\mathbf{K}}_t$ whose image contains $\Lambda_t$. Let $U$ be an open neighborhood of the image of $p$ such that $\overline{U}$ is diffeomorphic to a closed half hyperplane of $\R^4$. Let $V$ be the connected component of $\mathring{\mathbf{K}}_s\setminus f(\partial U)$ that contains $\Lambda_s$.

Of course, $N_t$ is diffeomorphic to $N_t\setminus\overline{U}$, and this is diffeomorphic via $f$ to $N_s\setminus \overline{V}$. But this contradicts Corollary~\ref{c:no embedding2}.
\end{proof}

\begin{remark}
In the above proof, it is clear that $V$ is homeomorphic to $\R^4$, but it is unclear whether it is exotic or standard; this ultimately depends on the smooth Schoenflies problem in dimension $4$. This is the unique obstruction to show that two non-diffeomorphic smooth $4$-manifolds $M$ and $N$ are non-diffeomorphic if and only if $M\setminus\Lambda$ and $N\setminus\Lambda$ are also non-diffeomorphic for any closed, totally disconnected, tame subset $\Lambda$. Previous Lemma can be seen as a version of this question for Taubes-like exotica.
\end{remark}

\begin{lemma}\label{l:nonisolated smoothings2}
Let $\mathbf{R}$ be a Taubes-like exotic $\R^4$. Let $\Lambda$ be any tame totally disconnected closed subset accumulating on the end of $\mathbf{R}$. Then the end of $\mathbf{R}\setminus\Lambda$ corresponding to the end of $\mathbf{R}$ is smoothly nonperiodic.
\end{lemma}
\begin{proof}
Suppose that $h:N_\mathbf{e}\setminus\Lambda\to N_\mathbf{e}\setminus\Lambda$ is a periodic embedding for the end $\mathbf{e}$, i.e., the sets $h^k(N_\mathbf{e}\setminus\Lambda_\mathbf{R})$ form a neighborhood base for $\mathbf{e}$. Since $\Lambda$ is tame, we can assume that $\Lambda \cap \partial N_\mathbf{e} = \emptyset$. Moreover, we can assume that $\partial N_\mathbf{e}$ is contained in $\psi^{-1}(S^3\times(r_\psi,\infty))$ for some admissible homeomorphism from a neighborhood of $\mathbf{e}$ to $[0,\infty)$.

Let $C$ be the  compact connected component of $\mathbf{R}$ bounded by $\partial N_\mathbf{e}$ and $h(\partial N_\mathbf{e})$. It follows that $$\mathbf{R}\setminus N_\mathbf{e}\cup_\partial C\cup_\partial C\cup_\partial\cdots$$ would be a periodic Taubes-like manifold by Lemma~\ref{l:extension}, which is not possible.
\end{proof}

\begin{remark}
As in Lemma~\ref{l:nonisolated smoothings}, the obstruction to extending Lemma~\ref{l:nonisolated smoothings2} to any nonperiodic exotic $\R^4$ is the smooth Schoenflies conjecture.
\end{remark}

\begin{definition}\label{d:nonisolated taubes-like}
A nonisolated end $\mathbf{e}$ of a noncompact smooth manifold is called {\em Taubes-like} if there exists a closed neighborhood $N_\mathbf{e}$ of that end and an embedding $i:N_\mathbf{e}\to Z$ for some $Z\in\MM$ with the following property:

If $\Sigma$ is a compact and connected submanifold of $N_\mathbf{e}$ such that it separates $\partial N_\mathbf{e}$ from $\mathbf{e}$ then $i(\Sigma)$ also separates the Taubes-like end of $Z$ from $\mathbf{K^\psi_0}$ for some admissible homeomorphism $\psi:X\to S^3\times[0,\infty)$, where $X$ is a neighborhood of the end of $Z$. 

If $Z\in\mathcal{R}$, then it will be said that the end is {\em Taylor-like}.
\end{definition}

\section{Exotic nonleaves}  

It is time to present the new examples of exotic nonleaves. They are appropriate smoothings of topological manifolds of the form $M\setminus\Lambda$ where $M$ is a closed simply connected topological $4$--manifold and $\Lambda$ is a tame totally disconnected compact subspace of $M$.

\subsection{$C^2$ exotic nonleaves}  

\begin{definition}[The family of $C^2$ nonleaves]\label{d:B(M,L)}
Suppose that $M$ is not homeomorphic to $S^4$. Let $\BB(M,\Lambda)$ be the smoothings of $M\setminus\Lambda$ such that at least one of its ends is not smoothly periodic and that end is diffeomorphic to at most finitely many other ends. 

In case $M$ is homeomorphic to $S^4$, a  smoothing $W$ of $S^4\setminus\Lambda$  belongs to $\BB(S^4,\Lambda)$ if and only
$W$ has a finite non-zero number of Taylor-like ends in $\MM_+$ (Definitions \ref{d:nonisolated taubes-like} and \ref{d:Taylor exotica}), and there exists an open set $A$ with the following properties:
\begin{itemize}
\item $A$ is a union of neighborhoods of every end which is not a Taylor like end in $\MM_+$.

\item Any compact subset $K\subset A$ can be smoothly embedded in a finite connected sum of $\C P^2$'s.
\end{itemize}
\end{definition}

\begin{remark}\label{r:B(M,L) nonempty}
It is easy to show that $\BB(M,\Lambda)$ is always nonempty. More precisely, in \cite[proposition 2.1] {Gompf}, Gompf shows that $M\setminus\{\ast\}$ always admits a Taubes-like end; then puncturing this smoothing by $\Lambda\setminus\{\ast\}$ gives an element in $\BB(M,\Lambda)$ for $M\neq S^4$. When $M=S^4$, just take an exotic $\R^4$ in $\mathcal{R}\cap \MM_+$ and delete $\Lambda\setminus\{\ast\}$. In these examples, the ends in $\Lambda\setminus\{\ast\}$ are all standard, so any compact set in a neighborhood of these ends can be smoothly embedded in $\R^4$ and therefore in $\C P^2$.
\end{remark}

\begin{corollary}\label{c:C2 uncountably smoothings}
The family $\BB(M,\Lambda)$ consists of uncountably many smoothings of $M\setminus\Lambda$.
\end{corollary}
\begin{proof}
The smoothings given in Remark~\ref{r:B(M,L) nonempty} have exactly one Taubes-like end and the other ones are standard. This Taubes-like end can be replaced by any other Taubes-like end in the $1$-parameter family of smoothings given by Taubes' theorem. When $M$ is homeomorphic to $S^4$ this $1$-parameter family is formed by Taylor-like smoothings by Remark~\ref{r:Taylor-like}. By Lemma~\ref{l:nonisolated smoothings} all these smoothings, parametrized by $t>r_\psi$, are pairwise non-diffeomorphic.
\end{proof}

The aim of this section is to prove the next Theorem, which extends the main result in \cite{Menino-Schweitzer}.

\begin{theorem}\label{t:C2 nonleaves}
Let $M$ be a  closed simply connected topological manifold and let $\Lambda$ be a nonempty, tame, closed and totally disconnected subset (finite or infinite). Then $\BB(M,\Lambda)$ consists of smooth manifolds that cannot be diffeomorphic to any leaf of any $C^2$ codimension one foliation on a compact manifold.
\end{theorem}

\begin{proposition}
Let $\FF$ be a codimension one $C^{1,0}$ foliation of a compact $5$--manifold $M$. If $\FF$ has a leaf $L$ diffeomorphic to some $W\in\BB(M,\Lambda)$, then $L$ is a proper leaf without holonomy. \label{p:proper}
\end{proposition}

\begin{proof}
Since $L$ is simply connected, it is a leaf without holonomy. Assume that $L$ is not proper.

First case: $M\neq S^4$. By Freedman's Theorem (see \cite{Freedman}), $H_2(M,\Z)$ is non-trivial. Moreover, since $M\setminus\Lambda$ is homeomorphic to $M\# (S^4\setminus\Lambda)$ there exists a compact set $C\subset W$ homeomorphic to $M\setminus D^4$ so that $H_2(C,\Z)=H_2(W,\Z)$ and $H_2(W\setminus C,\Z)=0$. By means of the Mayer-Vietoris exact sequence, it is easy to see that there is no compact subset in $L$ homeomorphic to $C$ and disjoint from it (the complete argument is analogous to that given in \cite[Proof of Proposition 2.14, First case]{Menino-Schweitzer}). Applying Reeb stability to $C$ there exists a product bifoliation $C\times (-1,1)$ embedded in $\FF$ such that $C\times\{0\}\hookrightarrow L$. If $L$ is non-proper then it accumulates on $C\times\{0\}$ transversely and therefore meets the embedded product foliation infinitely many times, so $L$ must contain infinitely many pairwise disjoint sets homeomorphic to $C$, but that is impossible.

Second case: $M=S^4$. Let $\mathbf{R}\in\mathcal{R}\cap\MM_+$ be a Taylor-like exotic $\R^4$ used to model some end $\mathbf{e}$ of $W$. Let $N_{\mathbf{e}}$ be a neighborhood of $\mathbf{e}$ and let $i:N_\mathbf{e}\to\psi^{-1}([0,\infty))$ be a smooth embedding satisfying the conditions of Definition~\ref{d:nonisolated taubes-like}. Let $\Sigma\subset N_\mathbf{e}$ be a topological bicollared $S^3$ such that $\psi(i(\Sigma))$ is homologous to $S^3\times\{0\}$ and $\Sigma\subset \psi^{-1}(S^3\times(r_\psi,\infty))$. Let $U$ be a small compact tubular neighborhood of $\Sigma$ such that $i(U)$ is still contained in $\psi^{-1}(S^3\times(r_\psi,t))$.

Applying Reeb stability to this annulus we get a product $C^{1,0}$ foliation $U\times(-1,1)$ embedded in $\FF$ so that $U\times\{0\}\subset L$ and each $U\times\{t\}$ is diffeomorphic to $U$ via the projection along leaves of the transverse foliation $\NN$. If $L$ is not proper then it must meet this product foliation infinitely many times. Thus $W$ contains infinitely many annuli $U_k$, $k\in\N$, diffeomorphic to $U$.

Since $\mathbf{R}\in\MM_+$, it follows by Lemma~\ref{p:no embedding} and Remark~\ref{r:no embedding2}, that $U_k$ cannot be smoothly embedded in the neighborhood $A$ of the ends that do not lie in $\mathcal{R}\cap\MM_+$, for otherwise it would embed in a finite connected sum of $\C P^2$'s.

Since the ends of $W$ in $\mathcal{R}\cap\MM_+$ are finite by hypothesis, passing to a subsequence if necessary, we can assume that the $U_k$'s converge to some fixed Taylor-like end $\mathbf{e}'$ in $\MM_+$. It follows that $L\subset\lim_{\mathbf{e}'}L$.

Let $\mathbf{E}\in\mathcal{R}\cap\MM_+$ be the exotic $\R^4$ in which a neighborhood of $\mathbf{e}'$ can be suitably embedded. Let $\varphi:X\to S^3\times [0,\infty)$ be a homeomorphism between a neighborhood of the end of $\mathbf{E}$ and a standard cylinder, and let $j:N\to X$ be a suitable smooth embedding from a neighborhood $N$ of $\mathbf{e}'$ into $X$.

It follows that the annuli $U_k$ are smoothly embedded in $\varphi^{-1}(S^3\times[0,t))$ and converge to its end. A priori, the $j(U_k)'s$ do not need to separate the ends of $\mathring{\mathbf{C}}_t^\varphi$. If infinitely many of the $j(U_k)$'s do not disconnect, it follows, by Lemma~\ref{l:taylor stimate}, that $\mathbf{C}_t^\varphi$ cannot be smoothly embedded in any spin closed $4$-manifold with hyperbolic intersection form but this contradicts the fact that the end is Taylor-like.

Therefore there exists an exotic compact cylinder $C$ in $\mathbf{C}^\varphi_t$ which is limited on its sides by two diffeomorphic copies of $U$, say $j(U_k)$ and $j(U_m)$, with the same relative orientation, so that $C$ is the union of $j(U_k)$, $j(U_m)$, and the connected component bounded by them in their complement. Let $K_\psi$ be the closure of the bounded connected component of $\mathbf{R}\setminus i(U)$. The cylinder $C$ can be used to produce a smoothly periodic cylindrical end in the Taubes-like manifold $K_\psi\cup_{\partial}C\cup_{h}C\cup_{h}\dots\,,$ where $h$ is just an  orientation-preserving diffeomorphism from $j(U_k)$ to $j(U_m)$ on the two sides of $C$. This is a Taubes-like manifold by Lemma~\ref{l:extension} with a periodic end, but that contradicts Taubes' Theorem~\ref{t:Taubes}.
\end{proof}

\begin{proof}[Proof of Theorem~\ref{t:C2 nonleaves}] If $W\in \BB(M,\Lambda)$ then, by hypothesis it has a nonperiodic end that is diffeomorphic to at most finitely many other ends of $W$. Then Proposition~\ref{p:proper} shows that a leaf diffeomorphic to $W$ must be proper without holonomy and Dippolito's criterion for a product neighborhood (Proposition~\ref{p:product_neighborhood} and Remark~\ref{r:rigidity reduction}) shows that $W$ is $1$-rigid. Then, $W$ satisfies the hypothesis of Theorem~\ref{t:C2 nonleaf criterion} and Remark~\ref{r:double cover}, therefore it is a $C^2$-nonleaf.
\end{proof}

\begin{remark}[Stein nonleaves]
An interesting corollary is the following result. It is shown in \cite[Chapter 11]{Gompf2} that there exist uncountably many Stein structures on $\C P^2\setminus\{\ast\}$; thus a cardinality argument like the one in Corollary~\ref{c:good small family} shows that almost all of them are nonperiodic, and therefore they form an uncountable subset of $\BB(\C P^2,\ast)$. The same argument applies to a larger family of examples. As a consequence of Theorem~\ref{t:C2 nonleaf criterion}, there exist uncountably many Stein manifolds which are not diffeomorphic to leaves of codimension one $C^2$ foliations.

It is interesting to note that $4$-dimensional Stein manifolds cannot have more than one end, since they are homeomorphic to the interior of a (possibly infinite) handlebody with no $3$- or $4$-handles \cite[Theorem 11.2.6]{Gompf2}.
\end{remark}

\subsection{$C^{1,0}$ exotic nonleaves}  

In this section the puncture set $\Lambda$ will always be infinite. Our examples of $C^{1,0}$ nonleaves will be subfamilies,  still uncountable, of the family $\BB(M,\Lambda)$ defined in the previous subsection. Recall that by Proposition~\ref{p:proper} we know that all the elements in our families are $1$-rigid.
We shall use a procedure to construct smooth manifolds, the infinite connected sum, which we now define.

\begin{notation}[Infinite connected sum]
Let $V$ be an oriented noncompact smooth manifold, let $T\subset V$ be a closed discrete subset, and let $\Psi: T\to \Diff(n)$ be a map to the family of $n$--dimensional oriented smooth connected manifolds, where $n=\dim V$. Let $\{D_t\mid\ t\in T\}$ be a family of pairwise disjoint standard closed disks in $V$ such that $t\in D_t$ and the union $\bigcup_t D_t$ is closed, and let $D_{\Psi(t)}$ be an arbitrary standard closed disk in $\Psi(t)$. The manifold $V\#\Psi$ is defined to be the boundary union of $V\setminus\bigcup_t \mathring{D_t}$ with $\bigsqcup_t\Psi(t)\setminus\mathring{D}_{\Psi(t)}$ by identifying the boundary components relative to each $t\in T$ so as to respect the orientations.
	
	If $V$ has boundary, then $\bigcup_t D_t$ is assumed to be contained in the interior of $M$. Observe that the diffeomorphism class of $M\#\Psi$ does not depend on the choice of the standard disks. \label{n:D_s}
\end{notation}

\begin{definition}[The family of $C^{1,0}$ nonleaves]\label{d:EE} 
First, suppose that $M$ is not homeomorphic to $S^4$. Let $S$ be an infinite and discrete subset of $\Lambda$. Then $\mathcal{L}(M,\Lambda)$ is the family of smoothings of $M\setminus\Lambda$ that are at most $k$--to--one at $S$ and such that $S$ consists of smoothly nonperiodic ends. 

In the other case, when $M$ is homeomorphic to $S^4$, a smooth manifold
 $W$ homeomorphic to $M\setminus \Lambda$ belongs to
 $\mathcal{L}(S^4,\Lambda)$ if and only if there exist subsets $F$ and $S$ of $\Lambda$ and an open set $A\subset W$ satisfying the following conditions: 
\begin{itemize}
\item  $F\cap\overline{S}=\emptyset$, $F$ is finite, $S$ is discrete and infinite,

\item The ends corresponding to $F$ are Taylor-like and belong to $\MM_+$,

\item $A$ is a union of neighborhoods of every end which does not lie in $\mathcal{R}\cap\MM_+$.

\item Any compact subset $K\subset A$ can be embedded in a finite connected sum of $\C P^2$'s.
 
\item $W$ is at most $k$--to--one at $S\cup F$, every end in $S$ is smoothly nonperiodic, and 
every end in $\Lambda\setminus F$ admits a neighborhood that can be embedded in $S^4$.
\end{itemize} 
\end{definition}
\begin{figure}\label{f:nonleaf construction}
\centering
\includegraphics[scale=0.24]{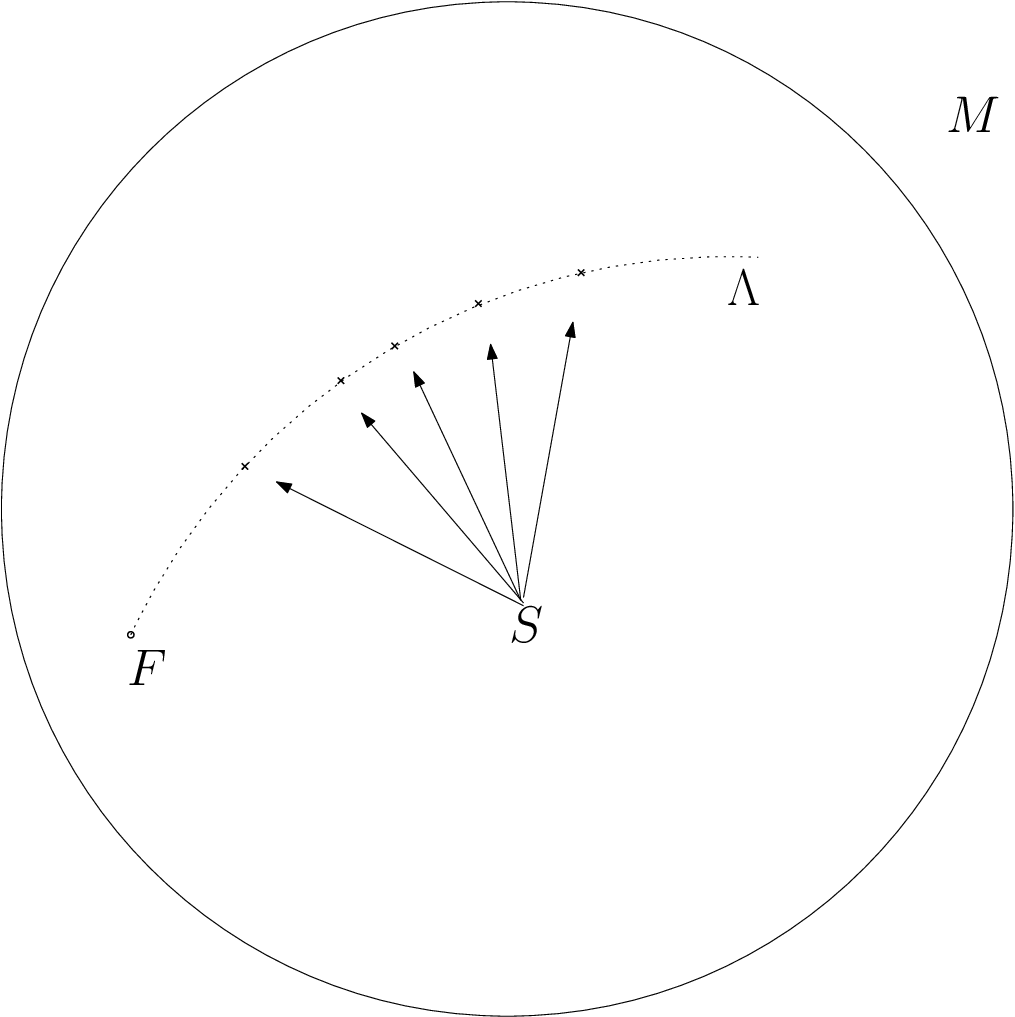}
\includegraphics[scale=0.24]{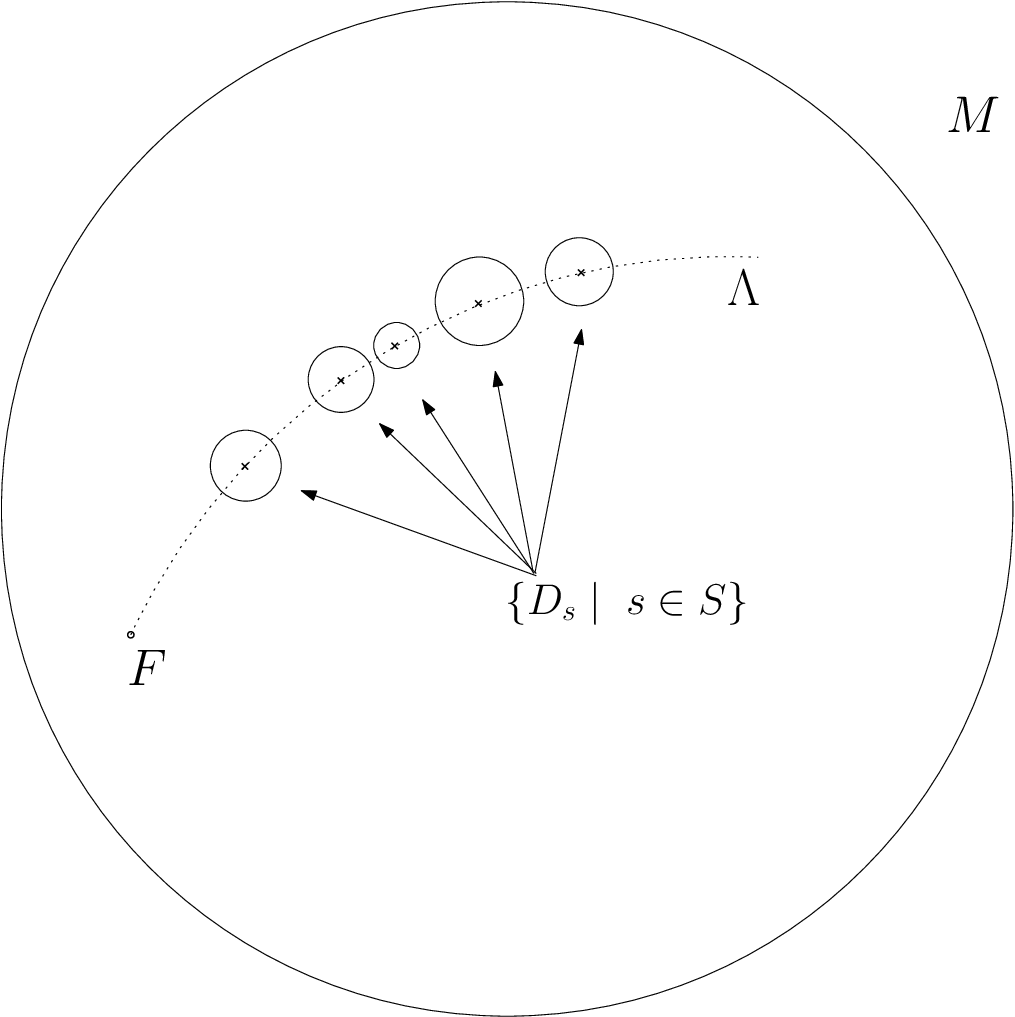}
\includegraphics[scale=0.24]{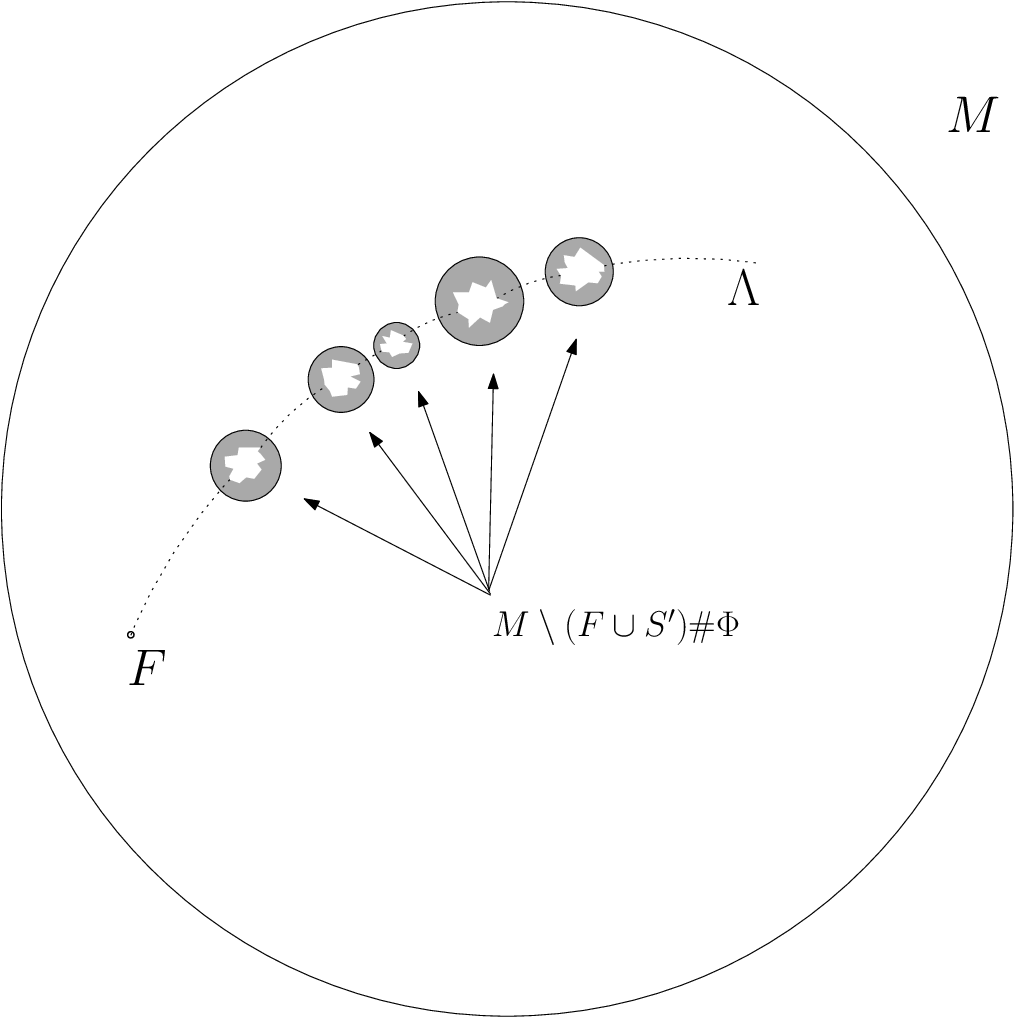}
\caption{Constructing manifolds in $\mathcal{L}(M,\LB)$; grey regions represent pairwise nondiffeomorphic Taubes-like ends (possibly non isolated) that can be embedded in $\C P^2$.}
\end{figure}

\begin{remark}\label{r:L(M,L) nonempty}
Is not difficult to show that $\mathcal{L}(M,\Lambda)$ is nonempty for every possible choice of an infinite, tame, closed and totally disconnected subset $\Lambda\subset M$. For instance, first take $F=\{\ast\}\subset \Lambda$ and take a Taubes-like smoothing of $M\setminus F$ (if $M=S^4$ take the smoothing to be in $\mathcal{R}\cap\MM_+$) as in Remark~\ref{r:B(M,L) nonempty}. Let $S\subset \Lambda$ be an arbitrary infinite and discrete set with derived set $S'$ and choose $S$ so that $\# S'=1$. Let $V$ be the smooth structure on $M\setminus (F\cup S')$ induced by the smoothing of $M\setminus F$. 

Choose $\mathbf{R}\in\mathcal{S}\cap\MM_-$ such that $\mathbf{R}$ can be smoothly embedded in $\C P^2$ (See Remark~\ref{r:no S4 - CP2 maybe}). Let $\psi$ be an admissible homeomorphism from a neighborhood of the end of $\mathbf{R}$ to $S^3\times [0,\infty)$. Let $\Phi:S\to \{\mathring{\mathbf{K}}^\psi_t\mid\ t>r_\psi\}$ be an injective map from $S$ into the $1$-parameter family of pairwise not end--diffeomorphic Taubes-like exotica associated to $\mathbf{R}$ and $\psi$. Set $\Lambda_1=\Lambda\setminus (F\cup \overline{S})$. It readily follows that $(V\#\Phi)\setminus\Lambda_1$ is at most one--to--one at $F\cup S$ by Lemma~\ref{l:nonisolated smoothings}. The ends corresponding to $F\cup S$ are smoothly nonperiodic by \ref{l:nonisolated smoothings2} (recall that these ends cannot be diffeomorphic to ends outside of $F\cup\overline{S}$ since the latter are standard and can be embedded in $S^4$) and they also cannot be diffeomorphic to the end in $S'$ since this end is limit of Taubes-like ends, a property that the ends in $F\cup S$ do not satisfy.

Let us check that this manifold belongs to $\mathcal{L}(M,\Lambda)$. This is trivial when $M\neq S^4$. For $M=S^4$, observe that all the ends in $S$ admit neighborhoods that can be smoothly embedded in $\C P^2$ (by the choice of $\mathbf{R}$). The ends outside of $F\cup\overline{S}$ are standard so they admit neighborhoods that can be smoothly embedded in $S^4$ and therefore in $\C P^2$. The end corresponding to $S'$ admits a neighborhood that can be smoothly embedded in the noncompact manifold $\C P^2\#\C P^2\#\C P^2\#\cdots$, so any compact set in this neighborhood can be embedded in a finite connected sum of $\C P^2$ as desired.

As in Corollary~\ref{c:C2 uncountably smoothings}, it can be shown that every $\mathcal{L}(M,\Lambda)$ consists of uncountably many distinct smoothings.
\end{remark}

\begin{proof}[Proof of Theorem~\ref{Thm1}]
All the elements in $\LL(M,\Lambda)$ are $1$-rigid and each element is at most $k$--to--one at a discrete and infinite subset of smoothly nonperiodic ends, for some $k\in\N$. Thus Theorem~\ref{t:nonleaf criterion} and Remark~\ref{r:double cover} implies that all these elements are $C^{1,0}$-nonleaves.
Then Remark~\ref{r:L(M,L) nonempty} shows that there exists a continuum of distinct smoothings in $M\setminus\Lambda$ which are $C^{1,0}$ nonleaves.
\end{proof}

\begin{remark}
Of course if $W\in\BB(S^4,\Lambda)$ then $\overline{W}\notin \BB(S^4,\Lambda)$ by construction but is still a $C^2$-nonleaf (or a $C^{1,0}$-nonleaf if $W\in\LL(S^4,\Lambda)$).

It is possible to obtain more topological types of nonleaves by extending the definition of Taubes-like ends. In this article we only consider cylindrical ends (or infinite connected sums of cylindrical ends), but the same arguments would apply to the case of (simply connected) admissible ends considered in \cite[Definition 1.3]{Taubes}. We avoid this technical generalization for the sake of readability.
\end{remark}

\section{Topological realization}\label{s:realization}  

It is unclear whether all the manifolds in our families of non-leaves considered in the previous section are {\em homeomorphic} to leaves of $C^{1,0}$ foliations, but it is clear that many of them can be so realized. An important concept for this section is that of a {\em properly proper leaf} (defined in \cite{Cantwell-Conlon4}); this just means that the proper leaf admits a closed transversal with just one intersection point.

Assume that the closed, totally disconnected set $\Lambda\subset M$ is tame in the closed $C^0$ $4$-manifold $M$
and has {\em finite type}, i.e., $\Lambda^{(k)}$ is finite and nonempty for some $k\in\N$ where $\Lambda^{(k)}$ is the $k$th derived set of $\Lambda$; this $k$ is called the {\em type} of $\Lambda$. In \cite{Cantwell-Conlon4} it is shown that any connected, oriented, noncompact surface whose end set has finite type can be realized as a leaf of a $C^\infty$ foliation in a closed $3$-manifold. If $\#\Lambda^{(k)}>1$ then we can assume that these leaves are properly proper. If $\#\Lambda^{(k)}=1$ then this latter requirement can be obtained if the regularity drops to $C^1$. 

The techniques used for that result in the specific case of $S^2\setminus\Lambda$ can be applied with obvious modifications to $S^4\setminus\Lambda$, for a tame $\Lambda$.

In \cite[Corollary 2.3]{Cantwell-Conlon4} it is shown that surfaces of finite type can be decomposed as finite connected sums of canonical noncompact blocks called $N_{r,i}$ and a compact surface $N_\ast$. If the surface is of the form $S^2\setminus \Lambda$ then $N_\ast = S^2$ and only blocks of the form $N_{r,0}$ are involved in the decomposition (the coordinate $i$ in $N_{r,i}$ refers to ends accumulated by genus). These canonical blocks $N_{r,0}$ are also spheres punctured along tame subsets and have analogous counterparts in any dimension, so $S^4\setminus\Lambda$ can be also be decomposed into canonical blocks, provided that $\Lambda$ is tame. 

Now each canonical block is realized as a properly proper leaf of a codimension one foliation on a closed manifold. Once this is done, trivial surgeries can be applied to realize the original manifold $M\setminus \Lambda$ as a leaf. Observe that we are not trying to control the ambient topology, so we can relax some of the definitions in \cite{Cantwell-Conlon4}; for instance, in our context, we change the notion of {\em availability} to {\em being realizable as a properly proper leaf}. The fundamental foliations given in \cite[Propositions 4.3 and 4.8]{Cantwell-Conlon4}, have immediate $4$-dimensional counterparts. With this in mind we get the following result.

\begin{proposition}\label{p:realization}
Let $\Lambda$ be a closed, totally disconnected, tame subset of $S^4$ with finite type $k$. Then $S^4\setminus\Lambda$ is diffeomorphic to a leaf of a $C^\infty$ codimension one foliation of a smooth closed $5$-manifold (depending on $\Lambda$). If $\#\Lambda^{(k)}>1$ then the leaf can be assumed to be proper with a closed transversal meeting the leaf in just one point. If $\#\Lambda^{(k)}=1$ then it can be realized as a proper leaf admitting such a closed transversal in a $C^{1}$ foliation. \end{proposition}

\begin{definition}
The set $\Lambda$ is said to be of {\em Cantor finite type} $k$ if $\Lambda^{(k)}$ is a Cantor set and $\Lambda^{(k-1)}$ is not a Cantor set. Set $\Upsilon=\Lambda\setminus \Lambda^{(k)}$ and let $e\in\Upsilon$; then there exists a maximum $r\in\N$, which we denote  by $t(e)$, such that $e\in\Lambda^{(r)}\setminus \Lambda^{(r+1)}$. If $e\in \Lambda^{(k)}$, $t(e)$ is defined as the maximum $r$ such that there exists a sequence $e_n\in\Upsilon$, $n\in\N$, converging to $e$ with $t(e_n)=r-1$, and $t(e)=0$ if no such sequence exists. The function $t:\Lambda\to\N$ will be called the {\em type function}, and  its restriction $t_0:\Lambda^{(k)}\to\N$ will be called the {\em Cantor type function}.

The set $\Lambda$ is said to have \em continuous Cantor finite type if it has Cantor finite type and the Cantor type function is continuous.
\end{definition}

\begin{lemma}\label{l:decomposition}
If $\Lambda$ has continuous Cantor finite type then there exists a finite clopen partition $\Lambda=\Lambda_0\sqcup\Lambda_1\sqcup\cdots\sqcup\Lambda_r$ where $\Lambda_0$ has finite type and $\Lambda_i$, $0<i\leq r$, has Cantor finite type with a constant Cantor type function.
\end{lemma}
\begin{proof}
Since $\Lambda$ has Cantor finite type, $\max t(\Lambda)\leq k$ for some $k\in\N$. Set $F_i=t_0^{-1}(\{i\})$, this forms a clopen partition of the Cantor set $\Lambda^{(k)}$. By continuity of the Cantor type function there exists a family of pairwise disjoint clopen neighborhoods $U_i$, $0\leq i\leq k$, of the $F_i$'s so that $\max t(U_i)= i$. It follows that $\Lambda_0=\Lambda\setminus\bigcup_i U_i$ has finite type. The family $\Lambda_0$ and the $U_i$'s form the desired partition of $\Lambda$.
\end{proof}

\begin{remark}
In the previous decomposition $\Lambda_0$ can be chosen as the empty set if and only if $\max t(\Lambda_0)\leq \max t(\Lambda^{(k)})$. Therefore, we shall always assume that $\max t(\Lambda_0)>\max t(\Lambda^{(k)})$ if $\Lambda_0\neq\emptyset$.

If $\Lambda$ has finite type then we shall also set $\Lambda_0=\Lambda$.
\end{remark}

\begin{proposition}\label{p:realization2}
Let $M$ be a smooth simply connected closed $4$-manifold. Let $\Lambda$ be a closed, totally disconnected, tame subset of $M$ with finite type $k-1$ or continuous Cantor finite type $k$. Then $M\setminus\Lambda$ is diffeomorphic to a leaf of a $C^\infty$ codimension one foliation of a compact $5$-manifold (depending on $\Lambda$). If $\#\Lambda_0^{(k-1)}>1$ then the leaf can be assumed to be properly proper. If $\#\Lambda_0^{(k-1)}=1$ then it can be realized as a proper leaf admitting such a closed transversal in a $C^{1}$ foliation. \end{proposition}
\begin{proof}
Observe that it suffices to consider the case $M=S^4$: assume that $S^4\setminus \Lambda$ can be realized as a properly proper leaf of some foliation $\FF$. Let $\tau$ be a closed transversal meeting that leaf in one point and remove a tubular neighborhood of that transversal. Perform a transverse gluing with the product foliation in $(M\setminus \mathring{D}^4)\times S^1$. The resulting foliation contains a properly proper leaf diffeomorphic to $M\setminus \Lambda$, as desired. 

Let $\Lambda=\Lambda_0\sqcup \Lambda_1\sqcup\dots\sqcup \Lambda_r$ be the clopen decomposition of $\Lambda$ given by Lemma~\ref{l:decomposition} (if $\Lambda$ has finite type then $\Lambda=\Lambda_0$). We shall show that the proposition works for each element of the decomposition.

For $\Lambda_0$ this is just given by Proposition~\ref{p:realization}. 

Assume now that $\Lambda$ has a constant Cantor type function. it follows that $t(\Lambda^{(k)})=\{k\}$. 

The construction is given by induction. Assume first that $k=0$, so $\Lambda$ is just a Cantor set. Consider a suspension of two hyperbolic Morse--Smale (analytic) diffeomorphisms of the circle generating a free group of rank $2$ over $2\#S^3\times S^1$ with a transverse Cantor minimal set. Every leaf meeting the gap of this Cantor set has the topology of $S^4\setminus\Lambda$. Hector's Lemma (see e.g. \cite{Cantwell-Conlon6} for its foliated version) shows that the holonomy on each gap is cyclic, therefore these leaves are proper. Moreover, by \cite[Lemma 3.3.7]{Candel-Conlon}, these are properly proper leaves as desired. Let us denote this foliation by $\FF_0$.

For $k=1$, let $\sigma_-$ and $\sigma_+$ be fibers of the previous suspension. These are closed transversals meeting every leaf infinitely many times. Moreover $(\sigma_-\cup\sigma_+)\cap L$ is a set of points accumulating on every end of $L$. Let $\FF_1$ be the foliation obtained from $\FF_0$ by turbulization of these fibers, removing the interiors of the resulting (generalized) Reeb components and indentifying the compact boundary components so that the transverse orientation is preserved. The endset of the leaves meeting any gap of the exceptional minimal set is homeomorphic to $\Lambda$ and are still properly proper as desired; moreover the compact leaf diffeomorphic to $S^3\times S^1$ is also properly proper.

For $k>1$, assume that a transversely oriented $\FF_{k-1}$ was constructed with just one properly proper compact leaf diffeomorphic to $S^3\times S^1$, with an exceptional minimal set so that the leaves meeting gaps are properly proper and homeomorphic to $S^4\setminus \Lambda'$. Let $\tau_-$ and $\tau_+$ be two disjoint closed transversals meeting the compact leaf in just one point and repeat the same process given for $k=1$. The resulting foliation $\FF_k$ contains a properly proper leaf homeomorphic to $S^4\setminus\Lambda$. This construction always produces $C^\infty$ foliations.

The general case is obtained as a combination of the previous ones. Let $\FF_{\Lambda_i}$, $0\leq i\leq r$, be a foliation having properly proper leaves homeomorphic to $S^4\setminus\Lambda_i$. Let $\sigma_0$ and $\sigma_1$ be closed transversals of $\FF_{\Lambda_0}$ and $\FF_{\Lambda_1}$ meeting a prescribed properly proper leaf in just one point. Remove a tubular neighborhood of these closed transversals and perform a transverse gluing in such a way that the circles bounding the prescribed leaves are identifyed. It follows that the resulting foliation contains a properly proper leaf diffeomorphic to $S^4\setminus \Lambda_0\# S^4\setminus\Lambda_1=S^4\setminus (\Lambda_0\cup\Lambda_1)$.

By finite induction we get the desired result. Observe that the regularity obstructions only depend on $\Lambda_0$, by Proposition~\ref{p:realization}.
\end{proof}

\begin{remark}\label{r:nonsmoothable case}
If $M$ is not smoothable, the above constructions work in the $C^0$ category, i.e. we can always obtain $C^0$ foliations with a properly proper leaf homeomorphic to $M\setminus \Lambda$. If $M$ is not smoothable with definite intersection form and $W$ is any smoothing of $M\setminus\Lambda$ then no open neighborhood of $\Lambda$ in $W$ can be smoothly embedded in $S^4$. Thus, it is not clear at all whether some of the previous $C^0$ foliations can be smoothed to $C^{1,0}$ ones. (Observe that our examples of $C^{1,0}$ non-leaves are just a subfamily of the possible smoothings of $M\setminus\Lambda$.)
\end{remark}

In the case where $M\neq S^4$, $H_2(M\setminus\Lambda,\Z)$ is nontrivial  (by Freedman's work) and finitely generated; thus, as we have seen in Proposition~\ref{p:proper}, $M\setminus\Lambda$ is $1$-rigid. If a smoothing of $M\setminus\Lambda$ is diffeomorphic to a leaf of a $C^2$ codimension one foliation, then the theory of levels \cite{Cantwell-Conlon} implies that it must lie at a finite level. Thus only end sets with finite type or Cantor finite type can occur in this case. However, it is not clear to us whether the Cantor type function must always be continuous. This leads to the following conjecture.

\begin{conjecture}
Let $L$ be a proper leaf of a codimension one $C^2$ foliation on a compact manifold. If $\xi(L)$ is not countable then it has Cantor finite and continuous type. 
\end{conjecture}

The combination of Theorem~\ref{Thm1} and Proposition~\ref{p:realization} gives a large family of exotic nonleaves, i.e., those which are homeomorphic but not diffeomorphic to leaves. As far as we know these are the first such examples.

When $\Lambda$ has neither finite nor continuous Cantor finite type and $M=S^4$ there is some hope of realizing $M\setminus \Lambda$ at the infinite level as a limit of leaves at finite level, in analogy with the main result of \cite{Cantwell-Conlon}. Unfortunately, the techniques of \cite{Cantwell-Conlon} cannot be adapted easily to handle the topology of $S^4\setminus\Lambda$, so this is also left as an open question. It seems likely that $M\setminus\Lambda$ can be homeomorphic to leaves of $C^{0}$ foliations (for general $M$ and $\Lambda$), but this would need a better understanding of $C^0$ codimension one foliations.

Finally, we want to explain why regularity obstructions appear in the cases where $\#\Lambda_0^{(k)}=1$. The $2$-dimensional version of the next result seems to be close to classical work about the topology of leaves of foliations in closed $3$-manifolds, see \cite[Theorem 3]{Cantwell-Conlon5} and \cite[Theorems 1 and 2]{Cantwell-Conlon4}. It also seems to be a corollary of level theory and some generalizations of Novikov's theorem, but we have not found any explicit reference, and therefore we have decided to include a proof.

\begin{proposition}
Let $(M,\FF)$ be a codimension one $C^2$ foliation on a closed $n$--manifold $M$, $n\geq 3$. If $L$ is a proper leaf and there exists a transverse circle to $\FF$ which meets $\lim L$, then $L$ is not homeomorphic to $\R^{n-1}$. Moreover, the only possible realization of $\R^{n-1}$ as a proper leaf is in a generalized Reeb component. \label{p:Novikov-Duminy-Kopell}
\end{proposition}
\begin{proof}
Assume that there exists a proper leaf $L$ homeomorphic to $\R^{n-1}$. Without loss of generality, assume that $\FF$ is transversely oriented. 

By Corollary ~\ref{c:end1connected} there exists a product neighborhood of $L$ foliated as a product. Let $\Omega$ be the connected component of the saturated open set formed by proper leaves homeomorphic to $L$ which contains $L$. This open set consists of proper leaves without holonomy, since its leaves are proper and simply connected. 

If $\Omega$ fibers over $\R$ then the boundary leaves of  $\widehat{\Omega}$ are homeomorphic to $\R^{n-1}$ (by Hector's Trivialization Lemma). If this is the case then, by definition of $\Omega$, this boundary leaf cannot represent a proper leaf of $\FF$. Thus it must be semiproper, but in $C^2$ regularity semiproper leaves have infinitely many ends by Duminy's theorem (see e.g. \cite{Cantwell-Conlon3}), so it cannot be homeomorphic to $\R^{n-1}$. Thus $\Omega$ fibers over the circle.

If the completion $\widehat{\Omega}$ is not compact, then at least one of its boundary leaves is not compact and any Dippolito decomposition has an arm. Since $\Omega$ fibers over the circle, it follows that the monodromy map is nontrivial. Thus, by means of the Kopell Lemma for foliations (see \cite{Cantwell-Conlon2}), it follows that some leaf in the interior of $\Omega$ has infinitely many ends, as explicitly shown in \cite[Lemma 2.20]{Menino-Schweitzer}. But $L$ has just one end, so  $\widehat{\Omega}$ must be compact. Now $\Omega\neq M$ since a compact manifold cannot be foliated completely by noncompact proper leaves, so $\widehat\Omega$ has at least one boundary leaf, and all its boundary  leaves are compact.

Since the leaves of $\FF_{|\Omega}$ are proper leaves without holonomy the holonomy group of each compact boundary leaf must be cyclic. Thus the holonomy covering of each compact boundary leaf is an open manifold with two ends and each boundary leaf is contained in $\lim L$, which consists of just one compact leaf, say $L_0$, since distinct compact leaves in $\lim L$ would define different ends of $L$. A neighborhood of that end agrees with that of one end of the holonomy covering of the compact boundary leaf $L_0$.

The cyclic holonomy is represented by an element in $H^1(L_0,\Z)$ and thus its Poincar\'e dual can be represented by a codimension one closed submanifold $B_0\subset L_0$ which lifts to nearby leaves in $\Omega$. Let us denote these lifts by $B_t\subset L_t$, for $0\leq t<\varepsilon$, where $L_t$ is the leaf passing through the point $\gamma(t)$ where $\gamma$ is a transverse arc with initial point in $B_0\subset L_0$ and its other points inside $\Omega$. Then for $t>0$ $B_t$ is null--homologous and bounds a compact subset of $L_t$, since each leaf in $\Omega$ is a euclidean space, while $B_0$ does not bound on $L_0$. Thus hypothesis (2) of Theorem 2.13 ($B_t$ is null--homologous and bounds a compact subset of $L_t$ for $t>0$, but $B_0$ does not bound on $L_0$) of \cite{Schweitzer1} is satisfied and $L_0$ must bound a generalized Reeb component. Thus $\lim L=L_0$ does not admit a closed transversal, contrary to hypothesis.
\end{proof}

\begin{remark}\label{r:Novikov-Duminy-Kopell}
The above proof applies in a wider setting: if $W$ is a simply connected one-ended manifold so that for any compact set $K\subset W$ there exists $K'\supset K$ such that $\pi_1(W\setminus K,K'\setminus K,x)=0$ for any $x\in K$ then it is not homeomorphic to a proper leaf of a $C^2$ codimension one foliation on a compact manifold with a transverse circle meeting its limit set. For instance, this applies to any simply connected manifold that is simply connected at infinity.
\end{remark}

\begin{ack}The first author wants to thank the Fundaci\'on Pedro Barri\'e de la Maza, Postdoctoral fellowship (2013--2015), the
CAPES postdoc program, Brazil (INCTMat \& PNPD 2015--2016), the
Ministerio de Economía y Competitividad, Grant MTM2014-
56950--P, Spain (2015--2017), the Programa Cientista do Estado do Rio de Janeiro, FAPERJ, Brazil (2015--2018), CNPq research grant 310915/2019-8 \& MathAmSud: ``Rigidity and Geometric Structures in Dynamics'' 2019-2020 (CAPES-Brazil) for their support during this reasearch.
\end{ack}


\begin{thebibliography}{00}
\bibitem{Alvarez-Barral} \'Alvarez L\'opez, J.A., and Barral Lij\'o, R. ``Bounded geometry and leaves.'' \textit{Math. Nachr.} 290, no. 10 (2017): 1448--1469.

\bibitem{Attie-Hurder} Attie, O., and Hurder, S. ``Manifolds which cannot be leaves of foliations.'' \textit{Topology} 35, no. 2 (1996): 335--353.

\bibitem{Bing} Bing, R.H. ``Tame Cantor sets in $E^3$.'' \textit{Pacific J. Math.} 11, no. 2 (1961): 435--446.

\bibitem{Candel-Conlon} Candel, A., and Conlon, L. \textit{Foliations I}, Graduate Studies in Mathematics 23, American Mathematical Society. Providence, Rhode Island, 2000.

\bibitem{Cantwell-Conlon6} Cantwell, J., and Conlon, L. ``Analytic Foliations and the Theory of Levels.'' \textit{Math. Ann.} 265 (1983): 253--261.

\bibitem{Cantwell-Conlon} Cantwell, J., and Conlon, L. ``Every surface is a leaf.'' \textit{Topology}  25, no. 3 (1987): 265--285.

\bibitem{Cantwell-Conlon2} Cantwell, J., and Conlon, L. ``Poincar\'e-Bendixon theory for leaves of codimension one.'' \textit{Trans. Am. Soc.} 265, no. 1 (1981): 181--209.

\bibitem{Cantwell-Conlon3} Cantwell, J., and Conlon, L. ``Endsets of Exceptional Leaves; A Theorem of G. Duminy.'' \textit{Foliations, Geometry and Dynamics, Warsaw} (2000): 225--261.

\bibitem{Cantwell-Conlon4} Cantwell, J., and Conlon, L. ``Leaf prescriptions for closed $3$-manifolds.'' \textit{Trans. Am. Soc.} 236 (1978): 239--261.

\bibitem{Cantwell-Conlon5} Cantwell, J., and Conlon, L. ``Leaves with isolated ends on foliated $3$--manifolds.'' \textit{Topology} 16 (1977): 311--322.

\bibitem{Cheeger} Cheeger, J. ``Finiteness theorems for Riemannian manifolds.'' \textit{Amer. J. Math.} 92 (1970): 61--74.

\bibitem{DeMichelis-Freedman} De Michelis, S., and Freedman, M.H. ``Uncountably many exotic $\R^4$'s in standard $4$-space.'' \textit{J. Diff. Geom.} 35 (1992): 219--254.

\bibitem{Dippolito}  Dippolito, P. ``Codimension one foliations of closed manifolds.'' \textit{Ann. of Math.} 107 (1978): 403--453.

\bibitem{Dippolito2} Dippolito, P. ``Corrections to 'Codimension one foliations of closed manifolds'.'' \textit{Ann. of Math.} 110 (1979): 203.

\bibitem{Donaldson} Donaldson, S.K. ``An application of gauge theory to four-dimensional topology.'' \textit{J. Diff. Geom.}  18, no. 2 (1983): 279--315.

\bibitem{Freedman} Freedman, M.H. ``The topology of four-dimensional manifolds.'' \textit{J. Diff. Geom.}  17 (1982): 357--453.

\bibitem{Freedman-Quinn} Freedman, M. H., and Quinn, F. \textit{Topology of $4$-Manifolds}, Princeton Math. series 39. Princeton University Press, 1990.

\bibitem{Ghys1} Ghys, E. ``Une vari\'et\'e qui n'est pas une feuille.'' \textit{Topology} 24, no. 1 (1985): 67--73.

\bibitem{Gompf}  Gompf, R.E. ``An exotic menagerie.'' \textit{J. Diff. Geom.} 37 (1993) 199--223.

\bibitem{Gompf2} Gompf, R.E., and Stipsicz, A.I. \textit{$4$-manifolds and Kirby calculus}, Graduate Studies in Mathematics 20, American mathematical Society Providence, Rhode Island, 1999.

\bibitem{Greene} Greene, R.E. ``Complete metrics of bounded curvature on noncompact manifolds.'' \textit{Arch. Math.} 31, no. 1 (1978): 89--95.

\bibitem{JAP} Inaba, T., Nishimori, T., Takamura, M., and  Tsuchiya N. ``Open manifolds which are non-realizable as leaves.'' \textit{Kodai Math. J.} 8 (1985): 112--119.

\bibitem{Janusz} Januszkiewicz, T. ``Characteristic invariants of noncompact Riemannian manifolds.'' \textit{Topology}  23, no. 3 (1984) 289--301.

\bibitem{Menino-Schweitzer} Meni\~no Cot\'on, C., and Schweitzer, P.A. ``Exotic open $4$-manifolds which are nonleaves.'' \textit{Geom. Topol.} 22, no. 5 (2018): 2791--2816.

\bibitem{Menino-Schweitzer2} Meni\~no Cot\'on, C., and Schweitzer, P.A. ``Non periodic leaves of codimension one foliations.'' Preprint (2020).

\bibitem{Phillips-Sullivan} Phillips, A., and Sullivan, D. ``Geometry of leaves.'' \textit{Topology} 20 (1981): 209--218.

\bibitem{Quinn} Quinn, F. ``Ends of maps. III: dimensions 4 and 5.'' \textit{J. Differential Geometry} 17 (1982): 503--521.

\bibitem{Schweitzer} Schweitzer, P.A. ``Surfaces not quasi-isometric to leaves of foliations of compact
3-manifolds'', Analysis and Geometry in Foliated Manifolds: Proc. VII Int. Colloquium on Differential Geometry, Santiago de Compostela (1995): 223--238.

\bibitem{Schweitzer1} Schweitzer, P.A. ``Riemannian manifolds not quasi-isometric to leaves in codimension one foliations'' \textit{Ann. Inst. Fourier} 61 (2011): 1599--1631.

\bibitem{Skora} Skora, R. ``Cantor sets in $S^3$ with simply connected complements.'' \textit{Topol. Appl.} 24 (1986): 181--188.

\bibitem{Sondow} Sondow, J.D. ``When is a manifold a leaf of some foliation?'' \textit{Bull. Amer. Math. Soc.} 81 (1975): 622--625.

\bibitem{Taubes} Taubes, C.H. ``Gauge theory on asymptotically periodic 4-manifolds.'' \textit{J. Diff. Geom.} 25 (1987): 363--430.

\bibitem{Taylor1} Taylor, L. ``An invariant of smooth $4$-manifolds.'' \textit{Geom. Topol.} 1 (1997): 71--89.

\bibitem{Zeghib} Zeghib, A. ``An example of a 2-dimensional no-leaf.''  \textit{ Geometric Study of Foliations} (1995): 475-477.
\end{thebibliography}
\end{document}